\documentclass[11pt, twoside, leqno]{amsart}  
\usepackage{soul}
\usepackage{lipsum}
\usepackage{amsfonts}
\usepackage{graphicx}   
\usepackage{epstopdf}
\usepackage{calligra}
\usepackage{amsfonts,amsmath,amsthm,amssymb}
\usepackage{mathtools}
\usepackage{hyperref}
\usepackage{autonum}
\usepackage{hhline}
\usepackage{array}
\usepackage[dvipsnames]{xcolor}
\usepackage{diagbox}
\usepackage{tcolorbox}
\usepackage{mdframed}
\usepackage{multicol}
\usepackage{graphicx}
\usepackage{subcaption}
\usepackage{moreverb}
\usepackage{bbm}
\usepackage[margin=1.34in]{geometry}
\usepackage{todonotes}
\usepackage{scalerel,amssymb}
\allowdisplaybreaks
\usepackage{mathrsfs}  
\usepackage{lineno}
\usepackage{todonotes}
\usepackage{tikz}
\usepackage[makeroom]{cancel}
\usepackage{appendix}
\usepackage{pgfplots}
\usetikzlibrary{arrows.meta}
\usepackage[numbers,sort&compress]{natbib}
\definecolor{mygreen}{HTML}{43a047}
\usepackage{subcaption}
\usepackage{doi}
\usepackage{alphalph}
\usepackage{booktabs}
\usepackage{makecell}
\usepackage{adjustbox}
\usepackage{enumitem}
\makeatletter

\@namedef{subjclassname@2020}{\textup{2020} Mathematics Subject Classification}
\newtheorem{theorem}{Theorem}
\newtheorem{lemma}{Lemma}
\newtheorem{proposition}{Proposition}
\newtheorem{example}{Example}

\newtheorem{assumption}{Assumption}

\newtheorem{remark}{Remark}

\numberwithin{lemma}{section}
\numberwithin{proposition}{section}
\numberwithin{theorem}{section}
\numberwithin{equation}{section}
\makeatother
\usepackage{array}
\newcolumntype{H}{>{\setbox0=\hbox\bgroup}c<{\egroup}@{}}

\newcommand{\Om}{\Omega}
\newcommand{\D}{\Delta}

\newcommand{\psitwo}{\psi_2^{\genk}}

\def \xin{\xi^{(n)}}
\def \xit{\xi^{(n)}_t}
\def \xitt{\xi^{(n)}_{tt}}

\def \bxin{\boldsymbol{\xi}}
\def \bxit{\boldsymbol{\xi_t}}
\def \bxitt{\boldsymbol{\xi_{tt}}}

\def \psit{\psi_t}
\def \psitt{\psi_{tt}}

\newcommand{\ddt}{\frac{\textup{d}}{\textup{d}t}}
\newcommand{\ddtn}{\frac{\textup{d}^n}{\textup{d}t^n}}

\newcommand{\bxi}{\boldsymbol{\xi}}

\newcommand{\ds}{\, \textup{d} s }
\newcommand{\dx}{\, \textup{d} x}

\newcommand{\intt}{\int_0^t}
\newcommand{\intT}{\int_0^T}

\newcommand{\R}{\mathbb{R}} 
\newcommand{\N}{\mathbb{N}} 

\newcommand{\Honezero}{H_0^1(\Omega)}

\newcommand{\bfq}{\boldsymbol{q}}

\newcommand{\genk}{\mathfrak{K}}
\newcommand{\tgenk}{\tilde{\mathfrak{K}}}
\newcommand\Lconv{\ast}

\newcommand{\mm}{\mathcal M}
\definecolor{grey}{rgb}{0.5,0.5,0.5}
 
\usepackage{stackengine,scalerel}
\newcommand\dhookrightarrow{\mathrel{\ThisStyle{\abovebaseline[-.6\LMex]{%
				\ensurestackMath{\stackanchor[.15\LMex]{\SavedStyle\hookrightarrow}{%
						\SavedStyle\hookrightarrow}}}}}}

\colorlet{brown}{brown!80!black}

\definecolor{darkgreen}{rgb}{0,0.5,0}

\newcommand{\alignedintertext}[1]{%
	\noalign{%
		\vskip\belowdisplayshortskip
		\vtop{\hsize=\linewidth#1\par
			\expandafter}%
		\expandafter\prevdepth\the\prevdepth
	}%
}

\makeatletter
\def\namedlabel#1#2{\begingroup
	\def\@currentlabel{#2}%
	\label{#1}\endgroup
}

\title[Energy decay of multi-term nonlocal MGT equations]{Energy decay of some multi-term nonlocal-in-time Moore--Gibson--Thompson equations}
  

\author[M. Meliani and B. Said-Houari]{Mostafa Meliani \and Belkacem Said-Houari}
     
\address{ 
	Department of Mathematics \\ 
	Radboud University   \\ 
	Heyendaalseweg 135,
	6525 AJ Nijmegen, The Netherlands}
\email{m.meliani@math.ru.nl} 

\address{ 
	Department of Mathematics \\ 
	College of Sciences\\
	University of Sharjah
	\\
	P. O. Box: 27272, Sharjah, UAE}
\email{bhouari@sharjah.ac.ae}

\usepackage{float} 
\begin{document}

\begin{abstract}
	This paper aims to explore the long-term behavior of some nonlocal high-order-in-time wave equations. These equations, which have come to be known as Moore--Gibson--Thompson equations, arise in the context of acoustic wave propagation when taking into account thermal relaxation mechanisms in complex media such as human tissue. While the long-term behavior of linear local-in-time acoustic equations is well understood, their nonlocal counterparts still retain many mysteries. We establish here a set of assumptions that ensures exponential decay of the energy of the system. These assumptions are then shown to be verified by a large class of rapidly decaying memory kernels. Under weaker assumptions on the kernel we show that one may still obtain that the    energy vanishes but without a rate of convergence. Furthermore, we refine previous results on the local well-posedness of the studied equation and establish a necessary initial-data compatibility condition for the solvability of the problem.
\end{abstract}
\maketitle

\section{Introduction}

Motivated by acoustic wave propagation in complex media exhibiting anomalous diffusion, we consider a linear equation of the form
\begin{subequations}\label{Main_System}
	\begin{equation}\label{eq:first}
		\begin{aligned}
			\tau (\genk\Lconv \psi_t)_{tt}+ \psi_{tt} - \gamma  \genk \Lconv \Delta \psi_t - c^2 \Delta \psi - \nu \Delta\psi_{t}   = 0,\quad \text{in}\quad \Omega\times (0,T)
		\end{aligned}
	\end{equation}
	where $\Omega\subset \R^d$ is a smooth bounded domain.  
	Above $\tau>0$ is the thermal relaxation time, $c>0$, the wave propagation speed, the constants $\gamma>0$ and $\nu>0$ can be understood as damping parameters.
	
	We supplement \eqref{eq:first} with the initial conditions 
	\begin{eqnarray}\label{Initial_Conditions}
		\big(\psi, \psi_t, (\genk\Lconv\psi_t)_t\big)\Big|_{t=0}  = (\psi_0,\psi_1,\psi_2^{\genk})  
	\end{eqnarray}
	and boundary condition:  
	\begin{eqnarray}\label{Boundary_Conditions}
		\psi|_{\partial \Omega}=0. 
	\end{eqnarray}
	The convolution terms in \eqref{eq:first} are defined as follows: for $g\in \{\psi_t, \Delta \psi_t\}$, we write   
	\begin{align}
		(\mathfrak{K}\Lconv g)(t)=\int_0^t  g(t-s)\,\mathfrak{K}(\ds), 
	\end{align} 
\end{subequations}
where we assume that the kernel $\mathfrak{K}$ is a Radon measure and has a resolvent $\tgenk$ (i.e.,  $\genk\Lconv\tgenk =1$) which can be written as the sum
\begin{equation}\label{Kernel_K}
	\tgenk = A\delta_0 + \mathsf{r} \qquad \textrm{with } A\in\R \, \textrm{ and } \ \mathsf{r} \in L^q_{\textup{loc}}(\R^+),\quad q>1.
\end{equation}

When $\genk$ is  the Dirac delta pulse $\delta_0$, equation~\eqref{eq:first} reduces to  the integer-order Moore--Gibson--Thompson (MGT) \begin{equation}\label{eq:MGT}
	\begin{aligned}
		\tau \psi_{ttt}+ \psi_{tt}  - c^2 \Delta \psi - (\gamma+\nu) \Delta \psi_t= 0.
	\end{aligned}
\end{equation}
This equation arises in  acoustics ~\cite{thompson,moore1960propagation} 
as the linearization of the so-called Jordan-Moore-Gibson-Thompson equation, an alternative model to the well-known Kuznetsov equation to describe the vibrations of thermally relaxing fluids, when the heat conduction of the medium is modeled using the Maxwell--Cattaneo law instead of the classical Fourier law; see for instance~\cite{KatLasPos_2012}. 
Surprisingly,  equation \eqref{eq:MGT} also arises in the modeling of the so-called standard linear solid model of viscoelasticity~\cite{Gorain_Bose_1998,Bose_Gorai_1998} and as a model for the heat conduction after introducing a relaxation parameter in the type III Green-Naghdi heat conduction model (see, e.g., \cite{quintanilla2019moore,conti2020thermoelasticity}). Furthermore, the very same MGT equation appears in the context of photo-acoustics under the name of Nachman--Smith--Waag equation, as well as in a fractional form that corresponds to \eqref{eq:first} (by setting, e.g., $\alpha_0 =0$ in~\cite[Equation (4.88)]{ammari2011mathematical}).

In the context of acoustics, the importance of studying \eqref{eq:first} stems from mounting evidence that acoustic propagation in biomedical tissue exhibits anomalous diffusion~\cite{parker2022power}. Thus, a newer trend in the field is to take into account more realistic heat exchange and dissipation laws in the derivation and analysis of the acoustic models.  We refer the reader to, e.g.,~\cite{holm2019waves,wismer1995explicit,holm2011causal} for various models accounting for the power-law attenuation of acoustic waves. Time-fractional MGT equations were first derived in \cite{kaltenbacher2022time} by modifying the entropy equation in the system of governing equations to capture anomalous heat exchange described by Compte--Metzler laws~\cite{compte1997generalized}. Specifically, equation \eqref{eq:first} arises when closing the set of governing equations of sound motion by considering the heat flux law 
\begin{equation}\label{eq:heat_flux}
	\bfq + \tau \genk \Lconv \bfq_t = -\kappa_1 \nabla \theta - \kappa_2 \genk \Lconv \nabla \theta,
\end{equation}
which describes the evolution of the heat flux $\bfq$ in the medium as a function of the temperature gradient $\nabla \theta$.
This law can be viewed as an averaging of two of the nonlocal Compte--Metzler laws proposed in~\cite{compte1997generalized}. 
The constants $\kappa_1$ and $\kappa_2$ are medium-dependent thermal conductivity coefficients and are related to the quantities $\nu$ and $\gamma-\tau c^2$ in \eqref{eq:first}. 

\subsection*{Related results for the local and nonlocal MGT equation}
Equation \eqref{eq:MGT} has been by now widely studied in the mathematical literature, see for  example~\cite{kaltenbacher2011wellposedness,lasiecka2015moore,dell2017moore,pellicer2019wellposedness,bucci2020regularity,Trigg_et_al,P-SM-2019} and the references contained therein.  In \cite{kaltenbacher2011wellposedness} the authors considered \eqref{eq:MGT} in a bounded domain and show that the linear dynamics is described by a strongly continuous
semigroup, which is exponentially stable provided the dissipativity
condition 
\begin{equation}\label{Stab_Cond}
	(\nu+\gamma)-\tau c^2>0
\end{equation}
is fulfilled. More precisely, they proved under the assumption \eqref{Stab_Cond} that the energy: 
\begin{equation}\label{Energy_Norm}
	\mathsf{E}[\psi](t)=\|\psi_t+\tau\psi_{tt}\|_{L^2}^2+\|\nabla (\psi+\tau\psi_t)\|_{L^2}^2+\|\nabla \psi_t\|_{L^2}^2
\end{equation}
decays at an exponential rate. The result in \cite{kaltenbacher2011wellposedness} was extended to the whole space $\R^d$ by Pellicer and Said-Houari in \cite{Pellicer:2021aa}, where they showed the well-posedness and investigated the decay rate of the solutions of \eqref{eq:MGT}.  They used the energy method in the Fourier space to show that under the assumption  \eqref{Stab_Cond}, the energy norm \eqref{Energy_Norm} satisfies 
\begin{equation}
	\mathsf{E}[\psi](t)\lesssim (1+t)^{-d/2}. 
\end{equation}   
We also refer to the paper by Chen and Ikehata  \cite{Chen_Ikehata_2021} for additional decay results.   

In the case $\nu=0$, the well-posedness analysis of \eqref{Main_System} with $\genk$ being the Abel fractional  derivative kernel:
\begin{equation}\label{Abel_Kernel}
	\genk =\frac{t^{\alpha-1}}{\Gamma(\alpha)}, \qquad  0<\alpha<1
\end{equation}
was discussed by Kaltenbacher and Nikoli\'c~in \cite[Section 7]{kaltenbacher2022time}. In particular, they showed that the equation is well-posed for initial data $(\psi_0,\psi_1,\psitwo)\in H^2(\Omega)\cap H_0^1(\Om)\times \{0\} \times L^2(\Om)$. In addition, they also studied the limit as $\alpha\to1^-$, and  showed  that the solution converges to that of the MGT equation \eqref{eq:MGT}. Parameter-asymptotic analysis of the model as $\tau \to 0^+$ has been the subject of study in \cite{meliani2023unified}. Specifically, it was shown that uniform-in-$\tau$ well-posedness is achieved for initial data in $(\psi_0,\psi_1,\psitwo) \in H_0^1(\Om)\times \{0\} \times L^2(\Om)$ and that as the thermal relaxation time $\tau$ tends  to $0$, the solution of \eqref{Main_System} converges to that of a fractionally damped second-order equation at a linear rate in $\tau$. However, no results are known on the long-term behavior of the nonlocal-in-time MGT equations.
\subsection*{Main Contributions}
Our first contribution   involves  demonstrating  well-posedness  provided that the initial data satisfy 
\[ (\psi_0,\psi_1,\psi_2^{\genk}) \in H_0^1(\Omega) \times L^2(\Omega) \times L^2(\Om)\] 
and the kernel $\genk$  is of  the form \eqref{Kernel_K}. (See Section~\ref{Sec:Main_Results} for the precise assumptions on $\genk$).  

While much progress has been made in recent years in this area~\cite{vergara2008lyapunov,vergara2015optimal,fritz2022equivalence}, many questions relating to the dissipation of energy of evolution equations with a nonlocal leading term are open. In particular, it is not clear what the best approach is to show the energy decay of higher-order-in-time multi-term nonlocal equations. 
Hence, after showing that solutions exist for all time $T>0$, we tackle the question of asymptotic behavior.
Using the energy method, we construct a suitable Lyapunov functional that leads to the exponential decay of the solution. Constructing such functionals requires introducing various functionals to capture the dissipation of different energy components.  The crucial step in the proof is identifying these functionals and appropriately combining them. 
We will also need to assume that the equation parameters satisfy
\begin{equation}\label{ineqs:condition_coefficients}
	\nu > 0 \quad  \textrm{and}
	\quad \gamma - \tau c^2>0.
\end{equation}
This condition is compatible with observations that have been made in \cite{Pellicer:2021aa,kaltenbacher2011wellposedness} for the MGT equation (for $\genk=\delta_0$) showing that 
the condition $\nu + \gamma -\tau c^2 >0$ is necessary for the long-term stability of the solution. The necessity for positivity of the individual components $\nu$ and $\gamma-\tau c^2$ stems from the need to show decay of essentially two quantities $\psi$ and $\genk\Lconv \psit$ (or to be more precise, their spatial and temporal derivatives), such that each of $\nu$ and $\gamma-\tau c^2$ contribute to achieving this goal.

It is worth mentioning here that acoustic and viscoelastic wave equations with nonlocal dissipation have been studied by a number of authors, see, e.g.,~\cite{conti2007decay,messaoudi2007global,lasiecka2015moore,hrusa1988model,munoz1996decay} and the references contained therein. These equations have a leading term of integer order and have, typically, the form
\[u_{tt} -c^2 \Delta u - \genk\Lconv \D u =0,\]
or, for their MGT counterparts~\cite{lasiecka2015moore,conti2021Moore}
\[\tau u_{ttt} + u_{tt}  -\gamma \Delta u_t -c^2 \Delta u - \genk\Lconv \D(\alpha u_t+ \beta u) =0.\]
While some of the ideas can be borrowed from these studies, their results do not inform us \emph{a priori} on the long-term behavior of \eqref{Main_System}. 

\subsection*{Organization of the paper}
The rest of the paper is organized as follows: First, we introduce the notations that we will use throughout as well as the relevant function spaces for the analysis. In Section~\ref{Sec:Main_Results}, we present  and discuss  the main results of the paper, which are given in the form of two theorems: Theorem~\ref{thm:wellp} and Theorem~\ref{thm:decay_Lyap}. In Section~\ref{Sec:Kernel_Verification}, we discuss the verification of the assumptions on the kernel needed for the well-posedness theory.  We also provide there a number of examples of kernels satisfying these hypotheses. In Section~\ref{Section_Local_Existence}, we provide the proof of Theorem~\ref{thm:wellp}. Section~ \ref{Section_Asymptotic_Behavior_1} is devoted to the asymptotic behavior of solutions. More precisely, we prove, under an appropriate assumption on the coefficients, that energy norms of the solution decay exponentially fast in time. In Section~\ref{sec:weak_decay}, we show that under weaker assumptions on the memory kernel, the vanishing of the energy can still be obtained, however, we will not be able to provide a decay rate. Attached to the current work are two appendices: Appendix~\ref{Sec:Preliminaries} is dedicated to proving certain compactness results on the space of solutions needed for the proof of Theorem~\ref{thm:wellp}. In Appendix~\ref{appendix:solvability_semi_discrete}, we give details about the solvability of the semi-discrete initial-boundary value problem  which is needed in the local well-posedness proof.
\subsection*{Notation}
Throughout the paper, we will use the following notations.
\begin{itemize}
	\item Below, we will use the notation $A\lesssim B$ for $A\leq C\, B$ with a constant $C>0$ that may depend on the spatial domain $\Omega$, which we assume bounded and Lipschitz-regular in $\R^d$.
	Here $d\geq1$ is the dimension of the space.
	\item Let $X$ and $Y$ be two Banach spaces. We write
	$X \hookrightarrow Y$ (respectively $X\dhookrightarrow Y$) for the continuous (respectively compact) embedding of $X$ into $Y$.
	\item In this work, $\mm(\R^+)$ stands for the space of finite Radon measures on $\R^+$, while $\mm_{\textup{loc}}(\R^+)$ is the space of finite measures on every compact subset of $\R^+$. For a subset $J\subset \R^+$, we will use $\|\cdot\|_{\mm(J)}$ to denote the total variation norm associated to $\mm(J)$ which is defined for $\genk\in\mm(J)$ as
	\[ \|\genk\|_{\mm(J)}  = \sup_{\phi \in C(J)} \left|\int_{J} \genk(ds) \phi (s) \right|\qquad \textrm{with} \  |\phi(s)|\leq 1 \textrm{ for all} \ \ s \in J;   \]
	see, e.g, \cite[Ch. 3 Theorem 5.4]{gripenberg1990volterra} and \cite[Theorem 6.19]{Rudin1987real}. Below, we will also often use $\|\genk\|_{\mm(\{0\})}$ to capture the presence of a point mass at $0$ in the kernel $\genk$ (in which case $\|\genk\|_{\mm(\{0\})} \neq 0$).
	
	\item We will use $\Lconv$ to denote the Laplace convolution, which should be interpreted as
	\begin{align}
		&(\mathfrak{K}\Lconv g)(t)=\int_0^t \mathfrak{K} (s)\,g(t-s)\ds \quad &&\textrm{if }\genk, g \in L^1(0,t),\\
		&(\mathfrak{K}\Lconv g)(t)=\int_0^t \mathfrak{K}(\textup{d}s)\,g(t-s) \quad &&\textrm{if }\genk \in \mm(0,t),\ g \in L^1(0,t).
	\end{align} 
	For the definition of the Laplace convolution of an integrable function and a Radon measure, we refer the reader to \cite[Ch. 3]{gripenberg1990volterra}.
	\item
	Let the final time $T>0$. Along the lines of \cite{meliani2023unified}, we define the Sobolev-like space
	\begin{equation} \label{def_Xfp}
		H^1 \cap X_\genk^\infty(0,T) = \{u \in L^\infty(0,T) \ |\ u_t\in L^2(0,T), \ \genk \Lconv u_t \in L^\infty(0,T)\},
	\end{equation}
	with the norm
	$$\|\cdot\|_{H^1\cap X_\genk^\infty(0,T)} = \max \big(\|u\|_{L^\infty(0,T)},\|u_t\|_{L^2(0,T)}, \|\genk\Lconv u_t\|_{L^\infty(0,T)} \big).$$
	Additionally, we define
	\begin{equation} \label{def_Yfp}
		Y_\genk^p(0,T) = \{u \in L^p(0,T) \ |\ (\genk \Lconv u)_t \in L^p(0,T)\}, \qquad \textrm{where } \, 1 \leq p \leq \infty,
	\end{equation}
	with the norm
	$$\|\cdot\|_{Y_\genk^p(0,T)} = \big(\|u\|_{L^p(0,T)}^p + \|(\genk\Lconv u)_t\|_{L^p(0,T)}^p \big)^{1/p},$$ with the usual modification for $p=\infty$. Completeness and sequential compactness questions relating to the space $Y_\genk^p$ were the subject of \cite[Lemma 3]{meliani2023unified}.
	
	Furthermore, the space
	\begin{equation}\label{V_R_Space}
		V_{\genk}(0,T) = \Big\{u\in L^2(0,T)\, \big|\, u_{t}\in L^2(0,T),\, (\genk\Lconv u_{t}+u)_{tt} \in L^2(0,T) \Big\},
	\end{equation}
	endowed with the norm 
	\begin{equation}\label{V_R_Norm}
		\|u\|_{V_{\genk}(0,T)} = \big(\|u\|_{L^2(0,T)}^2 + \|u_t\|_{L^2(0,T)}^2+  \|(\genk\Lconv u_{t}+u)_{tt}\|_{L^2(0,T)}^2\big)^{1/2}
	\end{equation}
	is important towards defining the space of solutions
	\begin{equation}
		H_\psi =  \Big\{ \psi \in H^1 \cap X_{\genk}^\infty(0,T; \Honezero)\cap V_{\genk}(0,T;H^{-1}(\Om))\,\Big|\, \psit \in Y_{\genk}^\infty(0,T; L^2(\Omega)) \Big\}.
	\end{equation}
\end{itemize}
\section{Main results}\label{Sec:Main_Results}
In this section, we state the main results of this paper. The well-posedness  result is summarized   in Theorem \ref{thm:wellp}, while the result on the decay estimate is stated in Theorem \ref{thm:decay_Lyap}.

\subsection{{Well-posedness}} 
In order to state the {result on well-posedness} of system \eqref{Main_System}, we make the following assumptions on the kernel $\genk$.

\begin{assumption}\label{Assumption_Kernel}	
	We assume that: 
	\begin{enumerate}
		\item[($\mathcal{A}0$) \namedlabel{assu0}{($\mathcal{A}0$)}] $\genk  \in \mm_\textup{loc}(\R^+).$
		
	\item[($\mathcal{A}1$) \namedlabel{assu1}{($\mathcal{A}1$)}]We assume that $\genk$ verifies $$\intt (\genk\Lconv y, y)_{L^2}\ds\geq 0
	$$ for all $t\in \R^+$ and $y\in L^2(0,t;L^2(\Om))$.

	\item[($\mathcal{A}2$) \namedlabel{assu2}{($\mathcal{A}2$)}]	Further, assume that 
	\begin{equation}
		\intt\big((\genk \Lconv y)_t, y\big)_{L^2}\ds \geq -C_{\mathcal{A}2} \|y(0)\|_{L^2(\Om)}^2
	\end{equation}
	for all $t\in \R^+$ and $y\in Y_\genk^2(0,t;L^2(\Om))$, with $C_{\mathcal{A}2} \geq 0$.
	\item[($\mathcal{A}3$) \namedlabel{assu3}{($\mathcal{A}3$)}] Additionally, we will require that $\genk$ has a resolvent measure, $\tgenk$ (with $\genk\Lconv\tgenk=1$), which can be written as the sum 
	\begin{equation}\label{K_form}
		\begin{aligned}
			&\tgenk = A\delta_0 + \mathsf{r} \qquad \textrm{with } A\in\R^+ \, \textrm{ and } \ \mathsf{r} \in L^q_{\textup{loc}}(\R^+)
			, \quad q>1,\\
		\end{aligned}
	\end{equation}
	where $\delta_0$ is the Dirac delta pulse at zero.  Additionally, if $\|\genk\|_{\mm(\{0\})} \neq 0 $, we assume that $\mathsf{r} \in W^{1,q}(0,T)$ for some $q>1$.
\end{enumerate}
\end{assumption}
A discussion on the verification of Assumptions~\ref{assu0}--\ref{assu3} is given in Section~\ref{Sec:Kernel_Verification} where we also provide some examples of kernels satisfying the above assumptions.

Now, we state the well-posedness theorem. 
\begin{theorem}[Wellposedness]\label{thm:wellp}
Let $T>0$, and let $\tau,\,\nu>0$ and $\gamma\geq \tau c^2$. 
Assume that {Set of Assumptions \ref{Assumption_Kernel}} holds.
Then, given initial data
\begin{equation}
	(\psi_0,\psi_1,\psi_2^{\genk})  \in \begin{cases}
		H_0^1(\Omega) \times L^2(\Omega) \times L^2(\Om)  \qquad \ \ \text{if }\quad  \|\genk\|_{\mm(\{0\})} =C_{\mathcal{A}_2} =0, \\[2mm]
		H_0^1(\Omega) \times H^1_0(\Omega) \times L^2(\Om) \qquad \ \text{if }\quad \|\genk\|_{\mm(\{0\})} \neq 0 \textrm{ or } C_{\mathcal{A}2} \neq 0, 
	\end{cases}
\end{equation}
verifying 
\begin{equation}\label{compatibility_condition}
	\psi_1 =A\psitwo \qquad \textrm{when} \qquad \|\genk\|_{\mm(\{0\})} =0,
\end{equation}
there is a unique $\psi \in H_\psi$ which solves 
\begin{equation}\label{eqn:ibvp_eqn}
	\begin{aligned}
		\begin{multlined}[t] 
			\intT \big<(\tau\genk * \psit + \psi)_{tt}, v\big>_{H^{-1}\times H^1}\ds 
			+ c^2  \intT(\nabla \psi, \nabla v)_{L^2}\ds \\
			+\gamma \intT\big(\genk \Lconv \nabla \psit , \nabla v\big)_{L^2}\ds + \nu \intT\big(\nabla \psit , \nabla v\big)_{L^2}\ds 
			= 0 
		\end{multlined}	
	\end{aligned}		
\end{equation}
for all $v\in L^2(0,T;H_0^1(\Omega))$, with
\[ \big(\psi, \psit, (\genk\Lconv\psi_t)_t\big) |_{t=0} = \big(\psi_0,\psi_1,\psitwo\big).\]
Furthermore, the solution verifies the following energy estimate:
\begin{equation}\label{ineq:energy_wellp_thm}
	\begin{multlined}
		\|\tau (\genk\Lconv\psi_{t})_t(t) \|^2_{L^{2}(\Omega)}+
		\|\psi_{t}(t)\|^2_{L^2(\Omega)}+  \|\tau \genk\Lconv\nabla\psi_t(t)\|^2_{L^2(\Omega)}+\| \nabla \psi(t)\|^2_{L^2(\Omega)} \\+ \nu \intt \|\nabla\psi_t\|_{L^2(\Om)}^2\ds \lesssim \|\tau \psitwo\|^2_{L^2(\Omega)}+\|\psi_1\|^2_{L^2(\Omega)}\\+ (\|\genk\|_{\mm(\{0\})}^2 + C_{\mathcal{A}2})\, \| \tau \nabla\psi_0\|^2_{L^2(\Omega)}+\|\nabla\psi_0\|^2_{L^2(\Omega)},
	\end{multlined}
\end{equation} 
for all $t\in(0,T]$. The constants $C_{\mathcal{A}2}$ and $A$ above are as in \ref{assu2} and \ref{assu3} respectively.
\end{theorem}

The proof of Theorem \ref{thm:wellp} is given in Section~\ref{Section_Local_Existence} and it is based on combining energy estimates with compactness properties of the space of solutions $H_\psi$. Before presenting the results on the asymptotic behavior of the solutions, we will briefly provide some comments and remarks related to Theorem \ref{thm:wellp}.  
\begin{enumerate}
\item[1.] Theorem \ref{thm:wellp} allows us to cover, among others, a large class of completely monotone kernels. We refer the reader to~\cite[Ch. 5, Theorem 5.4]{gripenberg1990volterra} for a result on the resolvent measures of completely monotone functions. 

\item[2.]The assumption on the compatibility of the initial data \eqref{compatibility_condition} is justified by the discussion on unique solvability of the semi-discrete problem in Appendix~\ref{appendix:solvability_semi_discrete}. 
For the fractional derivative kernels given by  \eqref{Abel_Kernel}, 
the compatibility condition \eqref{compatibility_condition} implies that $\psi_1 =0$ (this is due to its resolvent $\tgenk$ having no point mass at 0; see~\cite[Ch. 5, Theorem 5.4]{gripenberg1990volterra}). If we take the example of the exponential kernel 
$\genk =e^{-\beta t},$
with $\beta>0$, this condition dictates that $\psi_1 = \psitwo$ (this is due to its resolvent being $\tgenk = \delta_0 +\beta$). 
\item[3.] Note that under the stronger assumptions that  $\psit(0)=\psi_1=0$ and that $\genk$ has a regular enough resolvent $\tgenk$ (i.e., $A=0$  in \eqref{K_form}),
the well-posedness theory of system~\eqref{Main_System} as well as its $\tau$-limiting behavior follow from the analysis of \cite{meliani2023unified}.

\item[4.] {Set of Assumptions~}\ref{Assumption_Kernel} on the kernel is weaker than those made in \cite{meliani2023unified,nikolic2023nonlinear}. Thus, the cases of, e.g., the exponential kernel
\begin{equation}\label{Exponential_Kernel}
	\genk(t) = e^{-\beta t}, \qquad \beta>0
\end{equation}
or the polynomial kernel 
\begin{equation}\label{Polynomial_Kernel}
	\genk(t) = \frac{1}{(1+t)^{p}}, \qquad p>1
\end{equation}
satisfy {Set of Assumptions~}\ref{Assumption_Kernel} but not the assumptions in \cite{meliani2023unified,nikolic2023nonlinear}. Specifically, their resolvent contains a point mass at 0 (see, \cite[Ch. 5, Theorem 5.4]{gripenberg1990volterra}) and is thus not regular enough in the sense needed in \cite{meliani2023unified,nikolic2023nonlinear}.
This relaxation of the theoretical requirements regarding the regularity of the resolvent measure kernel is made possible by the presence of the $\nu$-term (with $\nu>0$), which gives us, after standard  
testing, control over $\|\nabla \psi_t\|_{L^2(0,T;L^2(\Om))}$ which we in turn leverage to consider ``nicer" solution time-spaces:
$$ H^1\cap X_\genk^\infty(0,T) = \{u \in L^\infty(0,T) \ |\ u_t\in L^2(0,T), \ \genk \Lconv u_t \in L^\infty(0,T)\}.$$
Indeed, the above-space is complete with no extra conditions on the regularity of the resolvent of $\tgenk$; see Lemma~\ref{Lemma:hybrid_Caputo_seq_compact} below.
\end{enumerate}

\subsection{Asymptotic behavior}
The second main result of this paper is to show the exponential decay of the energy
\begin{equation}\label{E_psi_def}
\begin{multlined}
	E[\psi](t):=\frac{1}{2} \Big(\|\tau (\genk\Lconv \psi_t)_t+\psi_t\|_{L^2(\Om)}^2+c^2\|\nabla(\tau \genk\Lconv \psi_t+\psi)\|_{L^2(\Om)}^2
	\\\hphantom{\|\psi_t\|_{L^2(\Om)}^2+c^2\|\nabla(\tau \genk\Lconv \psi_t\|_{L^2(\Om)}^2}+ \tau(\gamma-\tau c^2) \|\nabla \genk\Lconv \psi_t\|_{L^2(\Om)}^2\Big) \quad \textrm{for} \quad t\geq 0. 
\end{multlined}
\end{equation}
This will be done under stronger assumptions on the kernel than those made in {Set of Assumptions~}\ref{Assumption_Kernel}. The reason for that is the need to extract nice properties from the kernel, which will allow us to recover the dissipation for  some components of the above energy norm. 
It is clear from the definition of $E[\psi](t)$ that the assumption $\gamma-\tau c^2 \geq 0$ is needed for the coercivity of the energy functional. In fact, this is also needed to show that the above energy is dissipative, especially in the case $\nu=0$; see \eqref{ddt_Energy} below. 

\begin{assumption}\label{Assumption_Kernel_Exponential_Decay}
\begin{enumerate} Assume that:
	\item[($\mathcal{A}4$) \namedlabel{assu4}{($\mathcal{A}4$)}] There exist constants $\lambda_{sup}>0$ and  $\tilde{c}>0$ such that for all $0\leq\lambda\leq\lambda_{sup}$
	\begin{equation}
		\intt e^{\lambda s} (\genk\Lconv y, y)_{L^2}\ds\geq \tilde{c} \intt  e^{\lambda s}\|\genk\Lconv y\|_{L^2(\Om)}^2\ds,
	\end{equation}
	for all $t\in \R^+$ and $y\in L^2(0,t;L^2(\Om))$.
	\item[($\mathcal{A}5$) \namedlabel{assu5}{($\mathcal{A}5$)}]  Further, assume that 
	\begin{equation} 
		\intt e^{\lambda s} (\genk\Lconv y_t, y)_{L^2}\ds\geq  -\frac\lambda2 \int_0^t e^{\lambda s} (\genk\Lconv y, y)_{L^2}\ds
	\end{equation}
	for all $t\in \R^+$ and $y\in H^1\cap X_\genk(0,t;L^2(\Om))$ such that $y(0) = 0$. 
\end{enumerate}
\end{assumption}
Section~\ref{Sec:Verifying_exponential_kernels} provides a discussion on verifying Assumptions~\ref{assu4} and \ref{assu5}. Now, we state the decay estimate. 

\begin{theorem}[Exponential decay]\label{thm:decay_Lyap}
Suppose that Assumptions~\ref{Assumption_Kernel} and~\ref{Assumption_Kernel_Exponential_Decay} on the kernel hold. Assume in addition that  \eqref{ineqs:condition_coefficients} holds and that initial data are chosen as in Theorem~\ref{thm:wellp}, and that $\psi_1=0$. Then, there exists $\lambda>0$ such that for all $t\geq 0$
\begin{equation}\label{Integral_Ineq_L}
	E[\psi](t) \lesssim E[\psi](0) e^{-\lambda t} .  
\end{equation}
\end{theorem}
The proof of Theorem \ref{thm:decay_Lyap} is given in Section~\ref{Section_Asymptotic_Behavior_1}, and its idea is based on the construction of a suitable Lyapunov functional that is shown to decay exponentially under the kernel Assumptions~\ref{Assumption_Kernel_Exponential_Decay}. We briefly provide here some comments and remarks related to Theorem \ref{thm:decay_Lyap}.
\begin{enumerate}
\item[1.] The assumption that $\psi_1=0$ is essential in the nonlocal case ($\genk \neq \delta_0$) and allows us to write 
$$(\genk\Lconv\psi_t)_t= \genk\Lconv\psitt.$$
This is needed in order to be able to use the coercivity  assumption \ref{assu5} to control certain terms that will arise in the analysis; see, for instance, inequality~\eqref{why_psi1_zero} below.
We refer the reader to \cite[Ch. 3, Corollary 7.3]{gripenberg1990volterra} for details on the differentiation properties of convolution with finite Radon measures. Naturally, when $\genk =\delta_0$, the relationship above always holds and we do not have to require $\psi_1=0$.
\item[2.] In the case $\genk = \delta_0$, the condition \eqref{ineqs:condition_coefficients} can be weakened to 
\eqref{Stab_Cond},
which is a necessary condition for the asymptotic stability of \eqref{eq:MGT}. The reason for this is that in this case, $\genk\Lconv\psit$ and $\psit$ coincide, such that in, e.g, estimate~\eqref{E_Main_Estimate} only the sum $\nu+\gamma -\tau c^2$ need to be positive
(the constant $\tilde c =1$ in \ref{assu4} when $\genk =\delta_0$). Hence, our stability result covers the stability results obtained for the solution of \eqref{eq:MGT}. 
\item[3.] A finer decay result is given in Section~\ref{Sec:Finer_Decay_Result}, where we show decay of the individual quantities $\psit$ and $(\genk\Lconv \psit)_t$ (and not simply their sum as in the definition of the energy $E[\psi
]$); see Proposition~\ref{thm:decay_all_quants}.
\end{enumerate}

\section{Examples of covered kernels}\label{Sec:Kernel_Verification}
We give in this sections examples of kernels verifying {Set of Assumptions~}\ref{Assumption_Kernel} which we rely on to prove Theorem~\ref{thm:wellp}. One obvious example is given by:
\begin{example}\label{ex:dirac}
$\genk=\delta_0$, the Dirac pulse at 0. In which case, \eqref{eq:first} coincides with \eqref{eq:MGT}. For this kernel, \ref{assu0} is verified on account of $\|\genk\|_{\mm{(\R^+)}} =1$. Additionally Assumptions~\ref{assu1}--\ref{assu2} are verified and the constant in Assumption~\ref{assu2} is given by $C_{\mathcal{A} 2 } =1$. The resolvent $\tgenk =\mathsf{r}$ is the constant function $t \mapsto 1$ and fulfills the regularity assumptions of \ref{assu3}.
\end{example}

The Dirac delta pulse is not the only kernel that satisfies {Set of Assumptions~}\ref{Assumption_Kernel}. We provide in what follows a collection of results that allows us to verify {Set of Assumptions~}\ref{Assumption_Kernel} for certain completely monotone functions; see Examples~\ref{ex:exp}--\ref{ex:poly_decay} below.  

\subsection*{Verifying Assumption~\ref{assu1}}

Thanks to \cite{nohel1976frequency}, we know that \ref{assu1} holds, among others, for all kernels $\genk \in L^1_{\textup{loc}}(\R^+)$ such that $$(-1)^n\ddtn\genk\geq0 \quad \textrm {for} \quad n \in \{0,1,2\}.$$ 

Assumption~\ref{assu1} can alternatively be checked by means of Fourier transform, using Plancherel theorem. Indeed, \ref{assu1} is verified if the Fourier transform of $\genk$ is a.e non-negative.
\subsection*{Verifying Assumption~\ref{assu2}}
In order to verify Assumption~\ref{assu2} for completely monotone kernels, we prove the following lemma.
\begin{lemma}\label{lemma:Zacher_adapted}
Let $\genk \in L^1_{\textup{loc}}(\R^+)$ such that $$(-1)^n\ddtn\genk\geq0 \quad \textrm {for} \quad n \in \{0,1\},$$ 
and $y \in L^2(0,T;L^2(\Om))$. Then for $t>0$, the following holds
\begin{equation}
	\intt\big((\genk \Lconv y)_t, y\big)_{L^2}\ds \geq 0.
\end{equation}
\end{lemma}
\begin{proof}
The idea of the proof relies on successive approximations of $\genk$ using an approximating sequence 
$\{\genk_n\}_{n\in \N}\subseteq W^{1,\infty}(0, T)$ which verifies: 
\[\genk_n \geq 0, \qquad \frac{\textup{d}}{\textup{d}t} \genk_n \leq 0.\]
Such an approximating sequence is constructed along the lines of \cite[Lemma 5.1]{kaltenbacher2022limiting}. This allows exploiting the following identity:
\begin{equation}
	\big((\genk \Lconv y)_t,y\big)_{L^2}  = \frac12 \frac{\textup{d}}{\textup{d} t }\left(\genk \Lconv \|y\|_{L^2(\Om)}^2 \right) (t) + \frac 12 \frac{\textup{d}}{\textup{d} t } \genk(t) \|y(t)\|_{L^2(\Om)}^2 + \intt [-\frac{\textup{d}}{\textup{d}t}\genk (s)] \|y(t) - y(t-s)\|^2_{L^2(\Om)} \ds.
\end{equation}
See, e.g., \cite[Lemma 2.3.2]{zacher2010giorgi} and \cite[Lemma 3.1]{van2021existence}.

Integrating in time and using the assumption on the sign of $\genk_n$ and $\dfrac{\textup{d}}{\textup{d}t}\genk_n$ yields the desired result first for $\genk_n$, and by taking the limit $n\to\infty$, for $\genk$.

\end{proof}

\subsection*{Verifying Assumption~\ref{assu3}}
Assumption \ref{assu3} is designed to allow for smooth regular kernels. Indeed, if we were to assume the resolvent to be simply regular, say $\tgenk\in L^q(0,T)$ for some $q>1$ as was done in \cite{meliani2023unified}, then we would exclude large families of kernels. For example, it is known from \cite[Ch. 5, Theorem 5.4]{gripenberg1990volterra} that completely monotone kernels only admit a regular resolvent (i.e., $A=0$) if $\displaystyle\lim_{t\to0+}\genk =\infty$. When $\genk$ is continuous we can more precisely state
\[ A= \displaystyle\lim_{t\to0+} \frac 1 {\genk}. \]

The requirement that $q>1$ comes from the fact that the space $L^1(0,T)$ has poor sequential compactness properties as it is not reflexive~\cite[Ch. 4]{brezis2010functional}. Thus to prove attainment of the initial data we will rely on the weak sequential compactness of the unit ball in $L^q(0,T)$ for $q>1$  combined with Lemma~\ref{lem:continuous_embedding} in the course of the proof of Theorem~\ref{thm:wellp}. ~\\

Assumption~\ref{assu3} is verified, up to showing that $q>1$, among others, by completely monotone kernels according to~\cite[Ch. 5, Theorem 5.4]{gripenberg1990volterra}. Indeed, the result in \cite{gripenberg1990volterra} only guarantees the existence of resolvents that have an $L^1$-regular component $\mathsf{r}$. One may further construct examples where the requirement $q>1$ would not be satisfied along the lines of \cite[Example 2.2.2]{zacher2010giorgi}. However, in practice, we expect the requirement that $q>1$ is not overly restrictive and in fact will be shown to be verified for Examples~\ref{ex:exp}--\ref{ex:poly_decay} considered below.

In Appendix~\ref{App_sufficient_condition}, we furthermore provide a result which gives a sufficient condition for the resolvent $\tgenk = \mathsf{r} \in W^{1,q}(0,T)$ when $$\genk = B \delta_0 + \mathfrak{s},$$
with $B\in \R\setminus\{0\}$ and $\mathfrak{s} \in L^q(0,T)$ for $q> 1$.

From above, we conclude that completely monotone kernels verify \ref{assu0}--\ref{assu3} (up to showing that the resolvent $\tgenk$ has a regular enough component $\mathsf{r}$). In particular, we mention the following relevant kernel examples:

\begin{example}\label{ex:exp}
The classical exponential kernel \eqref{Exponential_Kernel} 
whose resolvent is given by $\delta_0 + \beta$. The function $t \mapsto \beta$ is a constant function and has thus the desired regularity.
\end{example} 
\begin{example}
The fractional derivative (Abel) kernel  \eqref{Abel_Kernel}, whose resolvent is given by $t \mapsto \dfrac{t^{-\alpha}}{\Gamma(1-\alpha)}$ and is thus in $L^{\frac{1-\varepsilon}{\alpha}}(0,T)$ for $0\leq \varepsilon<1-\alpha$ (with naturally $\dfrac{1-\varepsilon}{\alpha} >1$).
\end{example}  
\begin{example}\label{ex:Abel_reg}
The exponentially regularized Abel kernel (along the lines of that found in~\cite{messaoudi2007global})
$$\genk =\frac{t^{\alpha-1}e^{-\beta t}}{\Gamma(\alpha)},$$
with $0<\alpha<1,$ $\beta>0$. The regularity of the resolvent is ensured by combining the fact that through \cite[Ch. 5, Theorem 5.4]{gripenberg1990volterra}, {the resolvent is in $C^\infty(0,\infty)$}, with the local (in a neighborhoud of $0$) behavior of this kernel which follows that of the Abel kernel;
\end{example}  
\begin{example}
the fractional Mittag-Leffler kernels encountered in the study of fractional second order wave equations in complex media in~\cite{kaltenbacher2022limiting}
$$\genk =\frac{t^{\beta-1}}{\Gamma(1-\alpha)}E_{\alpha,\beta}(-t^\alpha),$$ 
with $0<\alpha\leq\beta\leq1$, where $E_{\alpha,\beta}$ is the two-parametric Mittag-Leffler function. The sufficient regularity of the resolvent is argued similarly to that of Example~\ref{ex:Abel_reg};
\end{example}  
\begin{example}\label{ex:poly_decay}
and the polynomially decaying kernel  \eqref{Polynomial_Kernel} arising in viscoelasticity, see, e.g.,~\cite{munoz1996decay}. The sufficient regularity of the resolvent is again argued by combining an argument on the global smoothness of the resolvent on $(0,\infty)$ and its local behavior near $0$.
\end{example} 

For the kernel Examples \ref{ex:exp}--\ref{ex:poly_decay}, the constant in Assumption~\ref{assu2} is given by $C_{\mathcal{A} 2 } =0$ thanks to Lemma~\ref{lemma:Zacher_adapted}.

\section{Proof of Theorem~\ref{thm:wellp}: Global existence of solutions}\label{Section_Local_Existence}
The aim of this section is to prove Theorem~\ref{thm:wellp}. Similarly to the study in \cite{meliani2023unified}, we take advantage of the fact that \eqref{eq:first} can be conveniently rewritten in the form
\begin{equation}\label{Wave_Equation}
(\tau (\genk\Lconv \psi_t)+ \psi)_{tt} - c^2\Delta ( \tau \genk \Lconv \psi_t +\psi) -(\gamma-c^2\tau)\genk\Lconv \Delta\psi_t- \nu \Delta\psi_{t}   = 0
\end{equation}
which is then effectively a wave equation for the unknown $\tau (\genk\Lconv \psi_t)+ \psi$ with extra terms that we intend to control using the assumptions on the kernel \ref{assu1} and \ref{assu2}.
\begin{proof}[Proof of Theorem \ref{thm:wellp}] 
	The well-posedness result follows from a standard Galerkin procedure along the lines of the one performed in \cite{meliani2023unified}. Details on the unique solvability of the semi-discrete problem are given in Appendix~\ref{appendix:solvability_semi_discrete}. To simplify the notation we omit the dependence of the solution $\psi$ on the Galerkin discretization parameter $n$ in the energy estimates below. 
	The proof will be divided into four steps.
	
	\begin{itemize}
		\item[(i)] We establish, on the approximate solutions given by the Galerkin procedure, \emph{a priori} estimates;
		\item[(ii)] We pass to the limit, thanks to compactness properties of the space of solutions. See Appendix~\ref{Sec:Preliminaries};
		\item[(iii)] We show that the solution obtained by the limiting procedure attains the initial data;
		\item[(iv)] We prove uniqueness of the solution in $H_\psi$.
	\end{itemize}
	
	\subsubsection*{A priori energy estimates}
	{In Step (i)}, we establish energy estimates on the approximate solutions obtained by the Galerkin method.
	Multiplying \eqref{Wave_Equation} by $\big(\tau \genk\Lconv \psi_t+ \psi\big)_t$ and integrating by parts over $\Omega$, we get 
	\begin{equation}\label{ddt_Energy}
		\begin{aligned}
			\ddt E[\psi](t)+\nu \|\nabla \psi_t\|_{L^2(\Om)}^2=&\,-(\gamma-c^2\tau) (\nabla \genk\Lconv \psi_t,\nabla \psi_t)_{L^2}\\
			&-\nu (\tau \nabla(\genk\Lconv \psi_t)_t, \nabla \psi_t )_{L^2},
		\end{aligned}  
	\end{equation}
	where the energy $E[\psi]$ is defined in \eqref{E_psi_def}. 
	Subsequently integrating in time on $(0,t)\subset(0,T)$ and using \ref{assu1} and \ref{assu2}, we obtain
	\begin{equation}\label{eq:energy_est_low}
		\begin{aligned}
			2 E[\psi](t)+ 2\nu \intt \|\nabla \psi_t\|_{L^2(\Om)}^2 \ds\leq&\,\|\tau\psitwo + \psi_1\|_{L^2(\Om)}^2+ c^2\| \nabla\psi_0\|_{L^2(\Om)}^2\\ 
			&+(c^2\tau^2  \|\genk\|_{\mm(\{0\})}^2 				  + \nu \tau C_{\mathcal{A}2})\, \| \nabla\psi_1\|_{L^2(\Om)}^2 .
		\end{aligned}  
	\end{equation}
	In order to be able to pass to the limit in the Galerkin procedure, we need to establish bounds for individual terms. To this end, consider the auxiliary time-evolution problem
	\[(\tau\genk \Lconv \psi_t)_t + \psi_t = z \quad \textrm{a.e. in } \Omega,\] with $z \in L^\infty(0,T;L^2(\Om))$. 
	Using \ref{assu3} and convolving with the resolvent measure $\tgenk$ yields the following Volterra integral equation of the second kind:
	\begin{equation}\label{eq:bootstrap_eq}
		(\tau+A) \psi_t+  {\mathsf{r}} \Lconv \psi_t = A z +  {\mathsf{r}} \Lconv z.
	\end{equation}
	Thus, according to existence theory of Volterra equations of the second kind~\cite[Ch. 2, Theorem 3.5]{gripenberg1990volterra}, equation \eqref{eq:bootstrap_eq} has a unique solution which satisfies:
	\begin{equation}
		\|\psi_t\|_{L^{\infty}(0,T;L^2(\Omega))} \lesssim \|z\|_{L^{\infty}(0,T;L^2(\Omega))},
	\end{equation}
	and subsequently 
	\begin{equation}
		\|\tau (\genk\Lconv \psi_t)_t\|_{L^{\infty}(0,T;L^2(\Omega))} \lesssim\|\psi_t\|_{L^{\infty}(0,T;L^2(\Omega))} + \|z\|_{L^{\infty}(0,T;L^2(\Omega))} 
		\lesssim \|z\|_{L^{\infty}(0,T;L^2(\Omega))}.
	\end{equation}
	If $\gamma > \tau c^2$, then a bound on $\|\genk\Lconv \nabla \psit(t)\|_{L^2(\Om)}$ is available in \eqref{eq:energy_est_low}, and the triangle inequality
	$$ \|\nabla \psi(t)\|_{L^2(\Om)} \lesssim \|\nabla(\tau (\genk\Lconv \psi_t)+\psi)\|_{L^2(\Om)}+ \|\nabla (\genk\Lconv \psi_t)\|_{L^2(\Om)}$$	
	allows us to get a bound  on $\|\nabla\psi(t)\|_{L^2(\Om)}$. If $\gamma = \tau c^2$, we can repeat the same argument as before by considering the auxiliary problem
	\[\tau\genk \Lconv \nabla\psi_t + \nabla\psi = y \quad\textrm{a.e. in } \Omega,\] with $y \in L^\infty(0,T;L^2(\Om))$.
	
	{Note that the right-hand side of 
		\begin{equation}
			(\tau (\genk\Lconv \psi_t)+ \psi)_{tt} = c^2\Delta ( \tau \genk \Lconv \psi_t +\psi) + (\gamma-c^2\tau)\genk\Lconv \Delta\psi_t + \nu \Delta\psi_{t}   
		\end{equation}
		is in $L^2(0,T;H^{-1}(\Omega))$.} This is thanks to the established estimates in \eqref{eq:energy_est_low} and the fact that the $H^{-1}(\Omega)$ norm of the Laplacian can be estimated as follows:
	\begin{equation}
		\|-\D u\|_{H^{-1}(\Omega)} =\sup_{\substack{v \in H_0^1(\Om) \\ \|v\|_{H_0^1(\Om)} \leq 1}} (-\D u, v)_{L^2(\Om)} = \sup_{\substack{v \in H_0^1(\Om) \\ \|v\|_{H_0^1(\Om)} \leq 1}} (\nabla u, \nabla v)_{L^2(\Om)} \leq \|\nabla u\|_{L^2(\Om)}.
	\end{equation}
	We then conclude that $\tau (\genk\Lconv \psi_t)_{tt} + \psitt \in L^2(0,T; H^{-1}(\Om))$.
	
	This implies that $\psi$ is bounded in $V_\genk(0,T; H^{-1}(\Om))$ and in turn 
	$(\genk\Lconv\psi_t)_{tt}$ and $\psi_{tt}$ are bounded in $L^p(0,T; H^{-1}(\Om))$ (with $p =\min \{2,q\}>1$)
	through the continuous embedding result given in Lemma~\ref{lem:continuous_embedding}. 
	
	\subsubsection*{Passage to the limit in the Galerkin procedure} The passage to the limit in the weak form is straightforward using the compactness properties of the spaces $Y_\genk^\infty$ and $H^1\cap X_\genk^\infty$, and $V_\genk$. Refer to \cite[Lemma 3]{meliani2023unified}, Lemma~\ref{Lemma:hybrid_Caputo_seq_compact}, and Lemma~\ref{lem:sq_compactness_highest_der} for respective compactness results.
	
	\subsubsection*{Attainment of initial data}
	In this part of the proof, we use the $\psi^n$ to emphasize the dependence of the solution at the semi-discrete level on the Galerkin discretization parameter $n$.
	
	Attainment of initial data is obtained by testing the weak form in the semi-discrete setting with $v\in C^\infty((0,T)\times \Omega)$ such that $v(T) =0$:
	\begin{equation}
		\begin{aligned}
			\begin{multlined}[t] 
				\intT \big<(\tau\genk * \psit^n)_{tt} , v\big>_{H^{-1}\times H^1}\ds + \intT \big<\psi^n_{tt}  , v\big>_{H^{-1}\times H^1}\ds + 
				c^2  \intT(\nabla \psi^n, \nabla v)_{L^2}\ds \\
				+\gamma \intT\big(\genk \Lconv \nabla \psit^n , \nabla v\big)_{L^2}\ds + \nu \intT\big(\nabla \psit^n , \nabla v\big)_{L^2}\ds 
				= 0, 
			\end{multlined}	
		\end{aligned}		
	\end{equation}
	where splitting of the first term is allowed given the combined smoothness and regularity of the test functions $v$ as well as the approximate solution $\psi^n$.
	Then using the idea of, e.g.,~\cite[Theorem 9.12]{salsa2016partial}, we integrate the first term in time and pass to the limit in the Galerkin discretization	using the sequential compactness of the unit ball in $L^p(0,T;H^{-1}(\Om))$ with $p>1$. This yields
	\begin{equation}
		\begin{aligned}
			\begin{multlined}[t] 
				-\intT \big<(\tau\genk * \psit)_{t} , v_t\big>_{H^{-1}\times H^1}\ds + \intT \big<\psi_{tt}  , v\big>_{H^{-1}\times H^1}\ds  
				+ c^2  \intT(\nabla \psi, \nabla v)_{L^2}\ds \\
				+\gamma \intT\big(\genk \Lconv \nabla \psit , \nabla v\big)_{L^2}\ds + \nu \intT\big(\nabla \psit , \nabla v\big)_{L^2}\ds 
				= - \big(\tau\psitwo , v(0)\big)_{L^2}.
			\end{multlined}	
		\end{aligned}		
	\end{equation}
	On the other hand, taking first the limit in the Galerkin procedure then integrating by parts in time in the first term yields
	\begin{equation}
		\begin{aligned}
			\begin{multlined}[t] 
				-\intT \big<(\tau\genk * \psit)_{t} , v_t\big>_{H^{-1}\times H^1}\ds + \intT \big<\psi_{tt}  , v\big>_{H^{-1}\times H^1}\ds  
				+ c^2  \intT(\nabla \psi, \nabla v)_{L^2}\ds \\
				+\gamma \intT\big(\genk \Lconv \nabla \psit , \nabla v\big)_{L^2}\ds + \nu \intT\big(\nabla \psit , \nabla v\big)_{L^2}\ds 
				= - \big((\tau\genk * \psit)_{t}(0) , v(0)\big)_{L^2},
			\end{multlined}	
		\end{aligned}		
	\end{equation}
	thus necessarily $\tau\psitwo=(\tau\genk * \psit)_{t}(0)$.
	The rest of the initial data are shown to be attained similarly. We omit the details here.
	\subsubsection*{Uniqueness} Uniqueness of the solution in $H_\psi$ is readily shown by assuming the initial data are given by $\psi_0 = \psi_1 = \psitwo = 0$. We want to show that necessarily $\psi = 0$ to prove uniqueness. 
	Following the approach of \cite[Theorem 7.2.4]{evans2010partial} where the linear wave equation is studied,
	we introduce valid test functions.
	Fix $ 0\leq t' \leq T$ and set 
	\begin{equation}
		w (t)=
		\left\{\begin{array}{ll}
			\int_t^{t'} \psi(s) \ds \quad  & \textrm{ if} \ \ 0\leq t\leq t'\\
			0 \quad  & \textrm{ if} \ \ t'\leq t\leq T.
		\end{array}\right.
	\end{equation} 
	We define the convolution-term analogous to $w$ as
	\begin{equation}
		w_{\genk} (t)= 
		\left\{\begin{array}{ll}
			\int_t^{t'} \genk\Lconv\psi_t(s) \ds \quad  & \textrm{ if} \ \ 0\leq t\leq t'\\
			0 \quad  & \textrm{ if} \ \ t'\leq t\leq T.
		\end{array}\right.
	\end{equation}  
	Both $w$, $w_{\genk} \in L^2(0,T;\Honezero)$ and are thus valid test functions in the sense of \eqref{eqn:ibvp_eqn}. The proof of uniqueness follows then similarly to \cite[Theorem 1]{meliani2023unified}. The details are omitted.
\end{proof}

\begin{remark}[On the case $\nu =0$]\label{rem:rem1} 
	It is possible to relax the condition $\nu > 0$ and allow for $\nu=0$. However, we would need $\genk$ to have a regular enough resolvent $\tgenk \in L^1_\textup{loc}(\R^+)$ so as to be able to use  the compactness result of~\cite[Lemma 1]{meliani2023unified} to pass to the limit in the Galerkin procedure. Additionally, to show attainment of the initial data, one would need to rely on the compact embedding result given in \cite[Lemma 2]{meliani2023unified} such that we would need $\tgenk \in L^q_\textup{loc}(\R^+)$ with $q>1$.
	In which case the unique solution of \eqref{Main_System} would exist in the space
	\begin{equation} 
		\{ \psi \in L^{\infty}(0,T; \Honezero)\,|\, \genk\Lconv\psit\in L^{\infty}(0,T; \Honezero), \ \psit \in Y_{\genk}^\infty(0,T; L^2(\Omega)) 
		\}
	\end{equation}
	similar to the statement of \cite[Proposition 2]{meliani2023unified}. 
\end{remark}

	\begin{remark}[$\tau$-uniformity of estimate \eqref{ineq:energy_wellp_thm}]
		The thermal relaxation parameter $\tau$ is small in practice and it is of interest to study the asymptotic behavior of the equation as $\tau \to 0^+$ along the lines of the results available in \cite{meliani2023unified,kaltenbacher2020vanishing,kaltenbacher2023vanishing}. The key then is establishing uniform-wellposedness estimates with respect to the parameter of interest. While the established estimate \eqref{ineq:energy_wellp_thm} holds for $\tau=0$, we note that we need $\tau+A \neq 0 $ in order to use Volterra equations existence theory to solve \eqref{eq:bootstrap_eq}.  In particular if $A=0$ (i.e., $\displaystyle\lim_{t\to0+}\genk =\infty$), the estimate is not $\tau$-uniform owing to the variation of constants formula in~\cite[Ch. 2, Theorem 3.5]{gripenberg1990volterra}.
	\end{remark}

Theorem~\ref{thm:wellp} provides an upper bound on the energy of the system by that of its initial state. One may even infer from \eqref{eq:energy_est_low} that since 
$$\lim_{T\to\infty}\nu \intT \|\nabla\psi_t\|^2_{L^2(\Om)}\ds<\infty,$$ 
then $\|\nabla\psi_t(t)\|_{L^2(\Om)} \to 0$ as $t\to \infty$ provided $\nu>0$. However, it neither guarantees vanishing of all components of the energy of the system nor provides a rate of decay. 

\section{Exponential decay of the energy}\label{sec:asympto_behav_1}\label{Section_Asymptotic_Behavior_1}
While Theorem~\ref{thm:wellp} provides local well-posedness for arbitrary final time $T>0$ with $\gamma \geq \tau c^2$, global well-posedness relies on the strict inequality $\gamma>\tau c^2$. This is to be expected from \cite{Pellicer:2021aa,kaltenbacher2011wellposedness} for the integer-order case. In the nonlocal case, we need to additionally assume $\psi_1 = 0$~\cite[Proposition 2]{meliani2023unified} (for the case $\nu=0$, but the testing there allows also for $\nu>0$). 

Our goal in this section is then to show, under the assumptions
\begin{equation}
	\nu > 0, \quad 
	\quad \gamma - \tau c^2>0, \quad  \textrm{and}\quad  \psi_t|_{t=0}=\psi_1=0,
\end{equation}
that the norm of some conveniently chosen unknowns converges to the steady state exponentially fast.
We prove this using the Lyapunov functional method. The idea is to introduce well designed functionals that are bounded by the energy norm and that capture the dissipation of some of the components of the energy \eqref{E_psi_def}. Through a weighted sum, we put these functionals together which   leads to a Lyapunov functional, denoted by $\mathcal{L} =\mathcal{L}(t)$, equivalent to the energy $E[\psi]=E[\psi](t)$, defined in \eqref{E_psi_def}
and which, we will prove, to satisfy  a differential inequality of the form 
\begin{equation}\label{Ineq_Lyap}
	e^{\lambda t}\mathcal{L}(t)+\int_0^t e^{\lambda s}E[\psi](s)\ds\lesssim E[\psi](0)\quad \textrm{for } t\geq 0, 
\end{equation} 
for some small positive constant $\lambda>0$. Clearly, proving \eqref{Ineq_Lyap} yields exponential decay of $E[\psi]$. Towards showing \eqref{Ineq_Lyap}, we need to extract enough decay from the kernel $\genk$. Specifically, we assume that \ref{assu4} and \ref{assu5} hold. We show in what follows how one can verify these assumptions for exponentially decaying kernels.

\subsection{Examples of kernels verifying Assumptions~\ref{Assumption_Kernel}--\ref{Assumption_Kernel_Exponential_Decay}}\label{Sec:Verifying_exponential_kernels}

We show here that we can construct a large family of satifactory kernels (with respect to {Set of Assumptions~}\ref{Assumption_Kernel_Exponential_Decay}). 
\begin{itemize}
	\item 
	Inspired by \cite{messaoudi2007global}, let $$\genk = k(t)\exp(-\beta t)$$ with $k \in L^1(\R^+)$, positive, nonincreasing, and convex, and let $\beta>0$. Then $\genk$ verifies \ref{assu4} for all $\lambda \leq 2\beta$. 
	Indeed, let $y\in L^2(0,t;L^2(\Om))$, we have
	\begin{align}
		\intt e^{\lambda s} (\genk\Lconv y, y)_{L^2}\ds &= \intt e^{\lambda s} \int^s_0 (k(s-\mu)e^{-\beta (s-\mu)} y(\mu), y(s))_{L^2}\textup{d}\,\mu\ds \\
		& = \intt \int^s_0 (k(s-\mu)e^{-(\beta-\lambda/2) (s-\mu)} e^{\lambda/2 \mu}y(\mu), e^{\lambda/2 s}y(s))_{L^2}\textup{d}\,\mu\ds \\
		&=\intt(k e^{-(\beta-\lambda/2) \, \cdot}\Lconv e^{\lambda/2\,\cdot}y,  e^{\lambda/2\,\cdot} y)_{L^2}\ds.
	\end{align}
	Notice that the kernel  $k(t)\exp(-(\beta-\lambda/2) t)$ verifies the assumptions of \cite[Theorem 2 (iii)]{staffans1976inequality}. We therefore obtain the inequality
	\begin{equation} \label{ineq:assu4_verif}
		\begin{aligned}
			\intt e^{\lambda s} (\genk\Lconv y, y)_{L^2}\ds &\geq \tilde c_\lambda \intt\|k e^{-(\beta-\lambda/2) \, \cdot}\Lconv e^{\lambda/2\,\cdot}y\|^2_{L^2}\ds \\ & =  \tilde{c}_\lambda \intt  e^{\lambda s}\|\genk\Lconv y\|_{L^2(\Om)}^2\ds,
		\end{aligned}
	\end{equation}
	with $\tilde c_\lambda>0$ depending on the kernel $k e^{-(\beta-\lambda/2)}$. The last equality follows from a direct computation.
	In the course of the proof of Proposition~\ref{thm:decay_all_quants}, we will have to adjust $\lambda$ so as to balance the energy of the system, therefore we need the coercivity constant $\tilde c_\lambda$ in \ref{ineq:assu4_verif}  to be made independent of $\lambda$. In order to ensure that, we rely on the fact that for all $0\leq\lambda\leq2\beta,$ we have
	\[0\leq k(t)e^{-(\beta-\lambda/2)t}\leq k(t) \in L^1(\R^+).\]
	Thus by Lebegue's dominated convergence theorem we get
	\[\|k(t)e^{-(\beta-\lambda/2)t} - k(t)e^{-\beta t}\|_{L^1(\R^+)}\xrightarrow[\lambda \to 0]{} 0.\]
	This ensures continuity of $\hat k_\lambda$, the Fourier transform of $k(t)e^{-(\beta-\lambda/2)t}$, with respect to the parameter $\lambda$. 
	Since 
	\[\tilde c_\lambda = \sup_{z\in \R} \, \frac{|\hat k_\lambda(z)|^2}{\Re \hat k_\lambda(z)},\]
	$\tilde c_\lambda$ depends continuously on $\lambda$. By picking 
	\[\tilde c = \min_{\lambda \in [0,2\beta]} \tilde c_\lambda,\]
	inequality \eqref{ineq:assu4_verif} yields 
	\begin{equation} 
		\intt e^{\lambda s} (\genk\Lconv y, y)_{L^2}\ds \geq \tilde{c} \intt  e^{\lambda s}\|(\genk\Lconv y)\|_{L^2(\Om)}^2\ds,
	\end{equation}
	for all $\lambda \in [0, 2\beta]$.
	
	\item 
	Similarly, writing
	\begin{align}
		\intt e^{\lambda s} (\genk\Lconv y_t, y)_{L^2}\ds &= \intt e^{\lambda s} \int^s_0 (k(s-\mu)e^{-\beta (s-\mu)} y_t(\mu), y(s))_{L^2}\textup{d}\,\mu\ds \\
		& = \intt \int^s_0 (k(s-\mu)e^{-(\beta-\lambda/2) (s-\mu)} e^{\lambda/2 \mu}y_t(\mu), e^{\lambda/2 s}y(s))_{L^2}\textup{d}\,\mu\ds \\
		& = \intt \int^s_0 (k(s-\mu)e^{-(\beta-\lambda/2) (s-\mu)} (e^{\lambda/2 \mu}y)_t(\mu), e^{\lambda/2 s}y(s))_{L^2}\textup{d}\,\mu\ds \\& \ \ - \frac \lambda2\intt e^{\lambda s} \int^s_0 (k(s-\mu)e^{-\beta (s-\mu)} y(\mu), y(s))_{L^2}\textup{d}\,\mu\ds,
	\end{align}
	we notice that the kernel
	$$t \mapsto k(t)\exp(-(\beta-\lambda/2) t)$$ verifies the assumption of \cite[Lemma 5.1]{kaltenbacher2022limiting} which yields \ref{assu5}.
\end{itemize}

Note that fractional kernels \eqref{Abel_Kernel} do not verify \ref{assu4} as per the Fourier transform analysis performed (with $\lambda =0$) in \cite{kaltenbacher2022limiting}.

The exponentially decaying kernels Example~\ref{ex:exp} and \ref{ex:Abel_reg} mentioned above fall under the category of tempered kernels which arise, e.g., when imposing a low frequency limit to fractional wave models; see~\cite{holm2023adding}.
\subsection{A first exponential energy inequality}

Our goal in this section is to construct a Lyapunov function that captures the dissipation of all the components in energy $E[\psi]= E[\psi](t)$. But first, let us state a lemma which gives us an energy inequality which we shall build on towards proving exponential decay of the energy. 
Below in Lemma \ref{Lemma_1_Decay}, we show that under the assumption $\gamma-\tau c^2>0$, the energy $E[\psi]$ (defined in \eqref{E_psi_def}) controls $ \|\nabla \psi\|_{L^2(\Om)}^2$.

\begin{lemma}\label{Lemma_1_Decay}
	Assume that $\gamma-\tau c^2>0$. Then, it holds that for all $t\geq 0$
	\begin{equation}\label{ineq:control_nabla_psit}
		\|\nabla \psi(t)\|_{L^2(\Om)}^2\lesssim E[\psi](t). 
	\end{equation}
\end{lemma}

\begin{proof}
	This can be shown by writing the identity $$ c\nabla \psi = c \nabla \psi + c \tau \nabla( \genk\Lconv \psi_t) - c \tau \nabla( \genk\Lconv  \psi_t)$$ and using the triangle inequlity to see that
	\begin{equation}
		\begin{aligned}
			c^2 \|\nabla \psi\|_{L^2(\Om)}^2  &\leq 2 c^2\|\nabla(\tau \genk\Lconv \psi_t+\psi)\|_{L^2(\Om)}^2+ 2 {\tau^2 c^2} \|\nabla \genk\Lconv \psi_t\|_{L^2(\Om)}^2 \\
			&\begin{multlined} \leq 2\Big(c^2\|\nabla(\tau \genk\Lconv \psi_t+\psi)\|_{L^2(\Om)}^2
				+ \tau(\gamma-\tau c^2) \|\nabla \genk\Lconv \psi_t\|_{L^2(\Om)}^2 \Big) \\+ 2  \tau (2\tau c^2 -\gamma)\|\nabla \genk\Lconv \psi_t\|_{L^2(\Om)}^2
			\end{multlined}
		\end{aligned}
	\end{equation}
	for all $t\geq0$.
	If $\gamma\geq 2 \tau c^2$, then $c^2 \|\nabla \psi\|_{L^2(\Om)}^2\leq 2 E[\psi]$ and we are done. In general, for $\gamma> \tau c^2$, we can write   
	\begin{equation}
		\begin{aligned}
			&c^2 \|\nabla \psi(t)\|_{L^2(\Om)}^2 \leq 2\Big(1 + \max\Big\{0,\frac{2\tau c^2 -\gamma}{\gamma-\tau c^2}\Big\}\Big) E[\psi](t) \quad \textrm{for}\quad  t\geq 0.
		\end{aligned}
	\end{equation}
	This yields \eqref{Lemma_1_Decay}. 
\end{proof}
\begin{lemma}\label{lem:main_est} 
	Suppose that Assumptions ~\ref{Assumption_Kernel}--\ref{Assumption_Kernel_Exponential_Decay} on the kernel hold. Assume in addition that 
	\eqref{ineqs:condition_coefficients} holds and that initial data are chosen as in Theorem~\ref{thm:wellp}, such  that $\psi_1=0$. Then, for all $t\geq 0$, we have  
	\begin{equation}\label{E_Main_Estimate}
		\begin{multlined}
			e^{\lambda t}E[\psi](t)+(\gamma-\tau c^2)\tilde{c}\int_0^t e^{\lambda s}\|\genk\Lconv \nabla \psi_t\|_{L^2(\Om)}^2\ds+ \nu \intt e^{\lambda s}\|\nabla \psi_t\|_{L^2(\Om)}^2\ds\\\lesssim E[\psi](0) + \lambda  \intt e^{\lambda s} E[\psi](s)\ds, 
		\end{multlined}
	\end{equation}
	for $0\leq\lambda\leq\lambda_{sup}$, where $\lambda_{sup}>0$ and $\tilde c>0$ are  given by the properties of the kernel through \ref{assu4} and \ref{assu5}. The hidden constant is independent of $t$. 
\end{lemma}
\begin{remark}
	Note that if $A \neq 0$ and $\|\genk\|_{\mm(\{0\})} = 0$, then setting $\psi_1 =0 $ imposes, through the statement of Theorem~\ref{thm:wellp} that $\psitwo =0 $ as well. 
\end{remark}
\begin{proof}[Proof of Lemma \ref{lem:main_est}]
	We intend here to test with $e^{\lambda t}\big(\tau \genk\Lconv \psi_t+ \psi\big)_t$. However, the term $\tau (\genk\Lconv \psi_t)_t$ lacks the necessary regularity in time to be a valid test function in the sense of \eqref{eqn:ibvp_eqn}. Therefore, similarly to the idea of, e.g., \cite[Theorem 8.3]{conti2021Moore}, we will carry out the testing in a  semi-discrete setting, such that the differential inequality~\eqref{E_Main_Estimate} will make sense at first for the Galerkin approximations of the equation. Taking the limit in the Galerkin procedure will then ensure that the inequality also holds for the solution of the infinite dimensional problem. 
	
	Thus, in a sufficiently regular Galerkin setting, multiplying \eqref{Wave_Equation} by $e^{\lambda t}\big(\tau \genk\Lconv \psi_t+ \psi\big)_t$  and integrating (by parts) over $\Omega$, we get 
	\begin{equation}\label{eq:Main_intermediate_1}
		\begin{aligned}
			e^{\lambda t}\frac{1}{2}\ddt E[\psi](t)+\nu e^{\lambda t} \|\nabla \psi_t\|_{L^2(\Om)}^2 + e^{\lambda t} (\gamma-c^2\tau) (\nabla (\genk\Lconv \psi_t),\nabla \psi_t)_{L^2}\\
			\hphantom{\nu e^{\lambda t} \|\nabla \psi_t\|_{L^2(\Om)}^2 + e^{\lambda t} (\gamma-c^2\tau) (\nabla (\genk\Lconv \psi_t),\nabla \psi_t)}+\nu e^{\lambda t}(\tau (\nabla(\genk\Lconv \psi_t)_t, \nabla \psi_t)_{L^2} =0.
		\end{aligned}  
	\end{equation}
	Integrating in time on $(0,t)$ the first term yields
	\begin{equation}
		\begin{multlined}
			\intt e^{\lambda s}\ddt E[\psi](s)\ds = 
			\intt \ddt e^{\lambda s} E[\psi](s)\ds 
			-\lambda \intt e^{\lambda s} E[\psi](s)\ds.
		\end{multlined}
	\end{equation}
	Using Assumption \ref{assu4}, we can treat the third term in \eqref{eq:Main_intermediate_1} 
	\begin{equation}
		\intt e^{\lambda s} (\gamma-c^2\tau) (\nabla \genk\Lconv \psi_t,\nabla \psi_t)_{L^2} \ds \geq (\gamma-c^2\tau) \tilde c\intt e^{\lambda s} \|\genk\Lconv \nabla\psi_t\|_{L^2(\Om)}^2 \ds.
	\end{equation}  
	While for the last term, we use the fact that $\psi_1=0$ (ensuring $\nabla(\genk\Lconv \psi_t)_t = \nabla(\genk\Lconv \psitt)$)
	and then \ref{assu5} to conclude
	\begin{equation}\label{why_psi1_zero}
		\begin{aligned}
			\intt \nu e^{\lambda s}(\tau \nabla(\genk\Lconv \psi_t)_t, \nabla \psi_t)_{L^2}\ds &=\, \intt \nu e^{\lambda s}(\tau \nabla\genk\Lconv \psitt, \nabla \psi_t)_{L^2}\ds \\ & \geq -\nu \frac\lambda2 \int_0^t e^{\lambda s} (\genk\Lconv \nabla \psi_t, \nabla \psi_t)_{L^2}\ds.
		\end{aligned}
	\end{equation}
	
	This yields the desired result in the semi-discrete setting. Taking the limit in the Galerkin procedure yields~\eqref{E_Main_Estimate}.
\end{proof}

\subsection{Constructing the Lyapunov functional}
If we are able to control the right-hand-side term $\lambda  \intt e^{\lambda s} E[\psi](s)\ds$, then inequality~\eqref{E_Main_Estimate} would give us exponential decay of the energy. However the left-hand-side of \eqref{E_Main_Estimate} only ensures exponential control of one component of $\intt e^{\lambda s} E[\psi](s)\ds$, namely,
$\intt e^{\lambda s}\|\nabla \genk \Lconv \psi_t\|_{L^2(\Om)}^2\ds$. To be able to control the other components of the energy, we need to define two functionals, the first of which is
\begin{equation}
	F_1(t)=\int_{\Omega} \Big((\tau (\genk\Lconv \psi_t)_t+\psi_t)(\tau \genk\Lconv \psi_t+\psi)\Big)(t)\dx \quad \textrm{for}\quad  t\geq 0. 
\end{equation}
This functional is designed to give us control over the component $\|\nabla(\tau \genk\Lconv \psi_t+\psi)\|_{L^2(\Om)}^2$ of the energy $E[\psi]$.
\begin{lemma}\label{lem:F1_est}
	Let the assumptions of Lemma~\ref{lem:main_est} hold. Then, for all $t\geq 0$, it holds that
	\begin{equation}\label{Est_F_1_E}
		|F_1(t)|\lesssim E[\psi](t) 
	\end{equation} 
	and 
	\begin{equation}\label{F_1_Main_Est_2}
		\begin{aligned}
			& e^{\lambda t }F_1(t)\,+\,(c^2-\varepsilon_1)\int_0^t e^{\lambda s } \|\nabla(\tau \genk\Lconv \psi_t+\psi)\|_{L^2(\Om)}^2\ds\\
			\lesssim& \,E[\psi](0)+\lambda \int_0^t e^{\lambda s} E[\psi](s)\ds \\
			&+ C(\varepsilon_1) \intt e^{\lambda s } \|\nabla \genk\Lconv \psi_t\|_{L^2(\Om)}^2 \ds+ \intt e^{\lambda s }\|\tau (\genk\Lconv \psi_t)_t+\psi_t\|_{L^2(\Om)}^2 \ds \\ &+C(\varepsilon_1 ) \intt e^{\lambda s } \|\nabla \psi_t\|_{L^2(\Om)}^2\ds
		\end{aligned}
	\end{equation}\\
	for $0\leq\lambda\leq\lambda_{sup}$ for some $\lambda_{sup}>0$. Here $\varepsilon_1>0$ is an arbitrary constant. The hidden constant is independent of $t$.
\end{lemma}
\begin{proof}
	Inequality \eqref{Est_F_1_E} holds 
	by using Young's inequality together with Poincar\'e's inequality. 
	Next, to prove \eqref{F_1_Main_Est_2},   we compute  the derivative of $F_1(t)$ with respect to time by making use of \eqref{Wave_Equation},  we obtain 
	\begin{equation}
		\begin{aligned}
			\ddt F_1(t)+c^2\|\nabla(\tau \genk\Lconv \psi_t&+\psi)\|_{L^2(\Om)}^2=\|\tau (\genk\Lconv \psi_t)_t+\psi_t\|_{L^2(\Om)}^2\\
			+&\,\int_{\Omega}\Big((\gamma-c^2\tau)\Delta(\genk\Lconv \psi_t)+ \nu \Delta\psi_{t}\Big)(\tau (\genk\Lconv \psi_t)+\psi)\dx.
		\end{aligned}
	\end{equation}  
	Again, this makes sense only in the semi-discrete setting, which we use as a regularization method for the argument. 
	Integrating by parts, we then obtain
	\begin{equation}\label{F_Est_1_2}
		\begin{aligned}
			&\ddt F_1(t)+c^2\|\nabla(\tau (\genk\Lconv \psi_t)+\psi)\|_{L^2(\Om)}^2\\
			=&\,\|\tau (\genk\Lconv \psi_t)_t+\psi_t\|_{L^2(\Om)}^2
			-\nu \int_{\Omega}\nabla \psi_{t}\cdot\nabla(\tau \genk\Lconv \psi_t+\psi)\dx\\
			& - (\gamma-c^2\tau)\int_{\Om}  (\genk\Lconv \nabla \psit) \cdot \nabla (\tau \genk\Lconv \psi_t+\psi)\dx .
		\end{aligned}
	\end{equation}
	We estimate the last two terms as  
	\begin{equation}
		\begin{multlined}
			\left| (\gamma-c^2\tau)\int_{\Om}  (\genk\Lconv \nabla \psit) \cdot \nabla (\tau \genk\Lconv \psi_t+\psi)\dx\right|\\ \hphantom{\left| \int_{\Om}   (\tau \genk\Lconv \psi_t)+\psi)\dx\right|}\leq \frac{\varepsilon_1}{2}\|\nabla(\tau (\genk\Lconv \psi_t)+\psi)\|_{L^2(\Om)}^2 +C(\varepsilon_1)\|\nabla (\genk\Lconv \psi_t)\|_{L^2(\Om)}^2 , 
		\end{multlined}
	\end{equation}
	and 
	\begin{equation}
		\left|\nu \int_{\Omega}\nabla \psi_{t} \cdot \nabla(\tau \genk\Lconv \psi_t+\psi)\dx\right|\leq\,\frac{\varepsilon_1}{2} \|\nabla(\tau \genk\Lconv \psi_t+\psi)\|_{L^2(\Om)}^2+C(\varepsilon_1 ) \|\nabla \psi_t\|_{L^2(\Om)}^2,
	\end{equation}
	where $\varepsilon_1>0$ is an arbitrarily small constant.
	
	Plugging the above two estimates into \eqref{F_Est_1_2}, we obtain  
	\begin{equation}\label{F_Est_1_3}
		\begin{aligned}
			&\ddt F_1(t)+(c^2-\varepsilon_1)\|\nabla(\tau \genk\Lconv \psi_t+\psi)\|_{L^2(\Om)}^2\\
			\leq &\,\|\tau (\genk\Lconv \psi_t)_t+\psi_t\|_{L^2(\Om)}^2
			+C(\varepsilon_1)\|\nabla \genk\Lconv \psi_t\|_{L^2(\Om)}^2+C(\varepsilon_1 ) \|\nabla \psi_t\|_{L^2(\Om)}^2. 
		\end{aligned}
	\end{equation}
	Multiplying \eqref{F_Est_1_3} by $e^{\lambda s}$ and integrating with respect to time on $(0,t)$, and using \eqref{Est_F_1_E}, leads to \eqref{F_1_Main_Est_2}.
\end{proof}

Next, to capture the dissipation term $\|\tau (\genk\Lconv \psi_t)_t+\psi_t\|_{L^2(\Om)}^2$, we define 
\begin{equation} \label{eq:functional_F2}
	F_2(t)=-\int_{\Omega}\tau \genk\Lconv \psi_t \, (\tau (\genk\Lconv \psi_t)_t+\psi_t)(t)\dx \quad \textrm{for}\quad  t\geq 0.
\end{equation}
The functional $F_2$ satisfies the following lemma.
\begin{lemma}\label{lem:F2_est}
	Let the assumptions of Lemma~\ref{lem:main_est} hold. Then, for all $t\geq 0$, it holds that
	\begin{equation}\label{Est_F_2_E}
		|F_2(t)|\lesssim E[\psi](t) 
	\end{equation} 
	and
	\begin{equation}\label{F_2_Main}
		\begin{aligned}
			&e^{\lambda t}F_2(t)+(1-\varepsilon_2)\int_0^t e^{\lambda s} \|\tau (\genk\Lconv \psi_t)_t+\psi_t\|_{L^2(\Om)}^2\ds \\
			\lesssim&\, E[\psi](0)+\lambda \int_0^t e^{\lambda s} E[\psi](s)\ds+ C(\varepsilon_2, \varepsilon_3) \intt e^{\lambda s}\|\nabla \psi_t\|_{L^2(\Om)}^2 \ds\\
			&+ C(\varepsilon_3) \intt e^{\lambda s} \|\nabla \genk\Lconv \psi_t\|_{L^2(\Om)}^2 \ds 
			+ \varepsilon_3 \intt e^{\lambda s} \|\nabla(\tau \genk\Lconv \psi_t+\psi)\|_{L^2(\Om)}^2\ds
		\end{aligned}
	\end{equation}
	for $0\leq\lambda\leq\lambda_{sup}$ for some $\lambda_{sup}>0$, and $\varepsilon_2, \varepsilon_3>0$ are arbitrary. The hidden constant is independent of $t$.
	
\end{lemma}
\begin{proof}
	The estimate \eqref{Est_F_2_E} is obvious. Next, to show \eqref{F_2_Main}, we take the derivative of $F_2$ with respect to $t$, using \eqref{Wave_Equation} we obtain 
	\begin{equation}
		\begin{aligned}
			\ddt F_2(t)=&\,-\int_{\Omega}\tau (\genk\Lconv \psi_t)_t (\tau (\genk\Lconv \psi_t)_t+\psi_t)\dx-\int_{\Omega}\tau (\genk\Lconv \psi_t) (\tau (\genk\Lconv \psi_t)_t+\psi_t)_{t}\dx\\
			=&\,-\int_{\Omega}|\tau (\genk\Lconv \psi_t)_t+\psi_t|^2\dx+\int_{\Omega}\psi_t (\tau (\genk\Lconv \psi_t)_t+\psi_t)\dx\\
			&+ \tau c^2\int_{\Omega}\nabla ( \tau \genk \Lconv \psi_t +\psi)\cdot \nabla(\genk\Lconv \psi_t)\dx +\tau\int_\Omega(\gamma-c^2\tau)|\nabla(\genk\Lconv \psi_t)|^2\dx\\
			&+\tau \nu \int_\Omega\nabla \psi_t \cdot \nabla(\genk\Lconv \psi_t) \dx. 
		\end{aligned}
	\end{equation}
	Applying Young's inequality and Poincar\'e's inequality, we have for any $\varepsilon_2, \varepsilon_3>0$ 
	\begin{equation}
		\begin{aligned}
			&\ddt F_2(t)+(1-\varepsilon_2)\|\tau (\genk\Lconv \psi_t)_t+\psi_t\|_{L^2(\Om)}^2\\
			\leq &\,C(\varepsilon_2, \varepsilon_3)\|\nabla \psi_t\|_{L^2(\Om)}^2+\varepsilon_3\|\nabla(\tau \genk\Lconv \psi_t+\psi)\|_{L^2(\Om)}^2+C(\varepsilon_3)\|\nabla \genk\Lconv \psi_t\|_{L^2(\Om)}^2. 
		\end{aligned}       
	\end{equation}
	Multiplying by $e^{\lambda s}$ then integrating with respect to time on $(0,t)$, and using \eqref{Est_F_2_E}, we obtain \eqref{F_2_Main}. 
\end{proof}

Taking advantage of Lemmas~\ref{lem:main_est}, \ref{lem:F1_est}, and \ref{lem:F2_est}, we are now in a position to define a suitable Lyapunov functional $\mathcal{L}$ as 
\begin{equation}\label{Lyapunov}
	\mathcal{L}(t):=N_0 E[\psi](t)+F_1(t)+N_1F_2(t) \quad \textrm{for}\quad  t\geq 0,
\end{equation} 
where $N_0$ and $N_1$ are large positive constants that will be fixed later.  

First, it is clear that for $N_0$ large enough, there exists two constants $c_1, c_2>0$ such that for all $t\geq 0$, it holds that 
\begin{equation}\label{Equiv_E_L}
	c_1E(t)\leq \mathcal{L}(t)\leq c_2E[\psi](t). 
\end{equation}
Indeed, we know from \eqref{Lyapunov}, and by using \eqref{Est_F_1_E} and \eqref{Est_F_2_E}, that 
\begin{equation}
	\begin{aligned}
		|\mathcal{L}(t)-N_0E[\psi](t)|\leq&\, |F_1(t)|+N_1|F_2(t)|\\
		\leq&\, c_0E[\psi](t), 
	\end{aligned}
\end{equation}
with $c_0>0$. Hence, for $N_0$ large enough, we can take $c_1=N_0-c_0$ and $c_2=N_0+c_0$ and then \eqref{Equiv_E_L} holds.
\subsection{Proof of Theorem~\ref{thm:decay_Lyap}}
We are now ready to prove Theorem~\ref{thm:decay_Lyap}.
Using \eqref{Lyapunov} together with the estimates \eqref{E_Main_Estimate},  \eqref{F_1_Main_Est_2} and \eqref{F_2_Main}, we obtain 
\begin{equation}  
	\begin{aligned}\label{eq:intermediary}
		&e^{\lambda t}\mathcal{L}(t)+\Big(N_0(\gamma-\tau c^2)\tilde{c}-C(\varepsilon_1)-N_1C(\varepsilon_3)\Big)\intt e^{\lambda s} \|(\genk\Lconv \nabla \psi_t)\|_{L^2(\Om)}^2\ds\\
		&+ \Big(N_0\nu- C(\varepsilon_1) - N_1 C(\varepsilon_2,\varepsilon_3)\Big) \intt e^{\lambda s}\|\nabla \psi_t\|_{L^2(\Om)}^2\ds\\
		&+(c^2-\varepsilon_1-N_1\varepsilon_3)\intt e^{\lambda s} \|\nabla(\tau \genk\Lconv \psi_t+\psi)\|_{L^2(\Om)}^2\ds\\
		&+\Big(N_1(1-\varepsilon_2)-1\Big)\intt e^{\lambda s} \|\tau (\genk\Lconv \psi_t)_t+\psi_t\|_{L^2(\Om)}^2\ds \lesssim E[\psi](0) + \lambda \intt e^{\lambda s} E[\psi](s)\ds,
	\end{aligned}
\end{equation}
for all $0<\lambda\leq \lambda_{sup}$, where the constant $\tilde c$ comes from \ref{assu4}.
We fix $\varepsilon_1$ and $\varepsilon_2$ small enough such that $\varepsilon_1<c^2$ and $\varepsilon_2<1$. Once $\varepsilon_2$ is fixed, we select $N_1$ large enough such that 
\begin{equation}
	N_1(1-\varepsilon_2)>1. 
\end{equation}
After that, we take $\varepsilon_3$ small enough such that
\begin{equation}
	\varepsilon_3<\frac{c^2-\varepsilon_1}{N_1}. 
\end{equation}
Once all the above constants are fixed, we take $N_0$ large enough such that the coefficients in the first two integral terms are positive.  
To finish, we pick 
\begin{equation}
	\begin{multlined}
		C^* \lambda = \min\Big\{N_0(\gamma-\tau c^2)\tilde{c}-C(\varepsilon_1)-N_1C(\varepsilon_3),N_0\nu- C(\varepsilon_1) - N_1 C(\varepsilon_2,\varepsilon_3),\\
		c^2-\varepsilon_1-N_1\varepsilon_3,N_1(1-\varepsilon_2)-1,C^*\lambda_{sup}\Big\}>0,
	\end{multlined}
\end{equation}
where $C^*$ is the hidden constant in \eqref{eq:intermediary} once the constants $\varepsilon_{1,2,3}>0$ and $N_{0,1}>0$ have been fixed.
Hence, by recalling \eqref{E_psi_def} and  using \eqref{Equiv_E_L}, we obtain \eqref{Integral_Ineq_L}. This ends the proof of Theorem \ref{thm:decay_Lyap}. 

\subsection{Decay of the pressure quantities: $\psit$ and $(\genk\Lconv \psit)_t$}\label{Sec:Finer_Decay_Result}
In the context of acoustics, $\psi$ in \eqref{eq:first} is the acoustic velocity potential. It is related to the acoustic pressure $p$ through the relationship
\[p = \rho \psit,\]
where $\rho$ is the medium density. 
Theorem~\ref{thm:decay_Lyap} shows the decay of $E[\psi]$. In turn, this guarantees decay of the `velocity' quantities $\nabla\psi$ and $\genk\Lconv\nabla\psit$ (the first is obtained by using the fact that $E[\psi]$ controls $c^2 \|\nabla \psi\|_{L^2(\Om)}^2$, see~\eqref{ineq:control_nabla_psit}). 
While it is possible that one obtains only decay of lumped up quantities and not of individual contributions (see, e.g., \cite{pellicer2019wellposedness}), we show here that we can also obtain decay of
of the following norm 
\begin{equation}\label{def:psi_mod}
	||\psi||_{\rm{mod}}:=\| (\genk\Lconv \psit)_t\|_{L^2(\Om)}^2 + \| \psit\|_{L^2(\Om)}^2+\| \nabla \genk\Lconv \psi_t\|_{L^2(\Om)}^2+\|\nabla\psi\|_{L^2(\Om)}^2. 
\end{equation}
\begin{proposition}\label{thm:decay_all_quants}
	Let the assumptions of Lemma~\ref{lem:main_est} hold. Then, there exists $\tilde\lambda>0$ such that for all $t\geq 0$
	\begin{equation}
		||\psi||_{\rm{mod}}\lesssim e^{-\tilde\lambda t} ||\psi(0)||_{\rm{mod}}.
	\end{equation}
\end{proposition}
\begin{remark}
	When $\genk=\delta_0,$ the result of Proposition \ref{thm:decay_all_quants} coincides with the classical exponential decay result of \cite{kaltenbacher2011wellposedness}. Here, we had to assume that $\psi_1=0$, but as explained before in the case of $\genk=\delta_0$ this assumption can be removed and the result of Theorem~\ref{thm:decay_Lyap} is valid for $$(\psi_0,\psi_1,\psitwo) \in H^1_0(\Om)\times H^1_0(\Om)\times L^2(\Om).$$
\end{remark}

To prove Proposition \ref{thm:decay_all_quants}, we first introduce the 
{following modified energy
	\begin{equation}\label{mod_energy} 
		\begin{multlined}
			E_\textup{mod}[\psi](t)= \|\tau ( \genk\Lconv \psi_t)_t+\frac{\tau c^2} \gamma\psi_t\|_{L^2(\Om)}^2 \\+ \frac{\tau c^2}{\gamma}\frac{\gamma-\tau c^2}\gamma\| \psit\|_{L^2(\Om)}^2+\frac{\gamma}\tau\| \tau\nabla (\genk\Lconv \psi_t)+\frac{\tau c^2}\gamma\nabla\psi\|_{L^2(\Om)}^2\quad \textrm{for} \quad t\geq 0.
		\end{multlined}
	\end{equation}
}
We show that by using the modified energy~\eqref{mod_energy} together with the one defined in \eqref{E_psi_def}, one can control all the unknowns of interest (i.e., $||\psi||_{\rm{mod}}$) when $\gamma>\tau c^2$. 

The Proposition \ref{thm:decay_all_quants}, will be given through three lemmas. 
\begin{lemma}\label{Equiv_Lemma}
	Assume that $\gamma>\tau c^2$. Then, there exists two positive constants $d_1$ and $d_2$ such that for all $t\geq 0$
	\begin{equation}
		d_1 ||\psi(t)||_{\rm{mod}}^2\leq E[\psi](t)+ {E_\textup{mod}}[\psi](t)\leq d_2 ||\psi(t)||_{\rm{mod}}^2.
	\end{equation}
\end{lemma}
\begin{proof}
	The inequality $E[\psi](t)+ {E_\textup{mod}}[\psi](t)\leq d_2 ||\psi(t)||_{\rm{mod}}^2$ is readily shown using triangle inequalities. 
	The first part of the inequality on the other hand relies on the fact that $E[\psi]$ controls the following components  $\|\nabla \genk\Lconv \psi_t\|_{L^2(\Om)}^2$ (from the definition in \eqref{E_psi_def}) and $\|\nabla \psi\|_{L^2(\Om)}^2$ (see Lemma~\ref{Lemma_1_Decay}). The component $\|\psi_t\|_{L^2(\Om)}^2$ is similarly controlled thanks to the definition of ${E_\textup{mod}}[\psi]$. Thus, there remains one component of $||\psi(t)||^2_{\rm{mod}}$ which needs to be controlled by ${E_\textup{mod}}[\psi]$. This is the subject of the next lemma.
\end{proof}

	\begin{lemma}
		Assume that $\gamma>\tau c^2$. Then, we have for all $t\geq 0$
		\begin{equation}\label{Est_psi_E}
			\|\tau ( \genk\Lconv \psi_t)_t\|_{L^2(\Om)}^2\lesssim {E_\textup{mod}}[\psi](t). 
		\end{equation}
	\end{lemma}
	
	\begin{proof}
		First observe that writing the identity 
		$$ \tau ( \genk\Lconv \psi_t)_t  = \tau ( \genk\Lconv \psi_t)_t + \dfrac{\tau c^2} \gamma\psi_t - \dfrac{\tau c^2} \gamma\psi_t,$$ 
		and applying the triangle inequality on the right-hand-side, one obtains
		\begin{equation}
			\begin{aligned}
				\|\tau ( \genk\Lconv \psi_t)_t\|_{L^2(\Om)}^2  \leq&\, 2 \|\tau ( \genk\Lconv \psi_t)_t+\frac{\tau c^2} \gamma\psi_t\|_{L^2(\Om)}^2 + 2 \Big(\frac{\tau c^2}{\gamma}\Big)^2 \|\psit\|_{L^2(\Om)}^2 \\
				\leq &\,2 \Big(\|\tau ( \genk\Lconv \psi_t)_t+\frac{\tau c^2} \gamma\psi_t\|_{L^2(\Om)}^2 + \frac{\tau c^2}{\gamma}\frac{\gamma-\tau c^2}\gamma \|\psit\|_{L^2(\Om)}^2  \Big) \\
				&+ 2  \frac{\tau c^2}{\gamma} \frac{2\tau c^2 -\gamma}{\gamma}\|\psit\|_{L^2(\Om)}^2\\
				&\begin{multlined}\leq 2 \Big(\|\tau ( \genk\Lconv \psi_t)_t+\frac{\tau c^2} \gamma\psi_t\|_{L^2(\Om)}^2 + \frac{\tau c^2}{\gamma}\frac{\gamma-\tau c^2}\gamma \|\psit\|_{L^2(\Om)}^2  \Big)\\ +2 \frac{2\tau c^2 -\gamma}{\gamma-\tau c^2}  \frac{\tau c^2}{\gamma} \frac{\gamma-\tau c^2}{\gamma}\|\psit\|_{L^2(\Om)}^2
				\end{multlined}
			\end{aligned}
		\end{equation}
		Recalling \eqref{mod_energy}, we deduce that 
		\begin{equation}
			\begin{aligned}
				\|\tau ( \genk\Lconv \psi_t)_t\|_{L^2(\Om)}^2 &\leq 2\left(1 + \max\Big\{0,\frac{2\tau c^2 -\gamma}{\gamma-\tau c^2}\Big\}\right) {E_\textup{mod}}[\psi](t),
			\end{aligned}
		\end{equation}
		which yields \eqref{Est_psi_E}. 
	\end{proof}
	\begin{lemma}\label{lem:main_est_2}
		Let the assumptions of Lemma~\ref{lem:main_est} hold. Then, for all $t\geq 0$, we have  
		\begin{equation}\label{E_Main_Estimate_2}
			\begin{multlined}
				e^{\tilde \lambda t}{E_\textup{mod}}[\psi](t)+\tau\frac{\gamma-\tau c^2}\gamma\tilde{c}\int_0^t e^{\tilde \lambda s}\|\genk\Lconv \psitt\|_{L^2(\Om)}^2\ds+ \nu \frac{\tau c^2} \gamma \intt e^{\tilde \lambda s}\|\nabla \psi_t\|_{L^2(\Om)}^2\ds\\\leq {E_\textup{mod}}[\psi](0) + \tilde \lambda  \intt e^{\tilde \lambda s} {E_\textup{mod}}[\psi](s)\ds, 
			\end{multlined}
		\end{equation}
		for $0<\tilde\lambda\leq\lambda_{sup}$ for some $\lambda_{sup}>0$. The hidden constant is independent of $t$. 
	\end{lemma}
	
	\begin{proof}
		To show \eqref{E_Main_Estimate_2}, we rewrite equation \eqref{Wave_Equation} as 
		\begin{equation}\label{eq:first_3}
			\begin{aligned}
				(\tau \genk\Lconv \psi_t+ \frac{\tau c^2}\gamma \psi)_{tt} + \frac{\gamma -\tau c^2} \gamma \psitt - \frac{\gamma}{\tau}\Delta ( \tau \genk \Lconv \psi_t + \frac{\tau c^2}\gamma\psi) - \nu \Delta\psi_{t}   = 0,
			\end{aligned}
		\end{equation}
		which we want to test with $e^{ \tilde \lambda s}(\tau (\genk\Lconv \psi_t)_t+\frac{\tau c^2} \gamma\psi_t)$.
		Here, for the same reasons put forth in the proof of Lemma~\ref{lem:main_est}, we carry out the testing in the semi-discrete setting and obtain
		\begin{equation}
			\begin{aligned}
				&e^{\tilde \lambda t}\frac{1}{2}\ddt {E_\textup{mod}}[\psi]+\nu \frac{\tau c^2} \gamma e^{\tilde \lambda t} \|\nabla \psi_t\|_{L^2(\Om)}^2 + e^{\tilde \lambda t} \tau\frac{\gamma-\tau c^2}\gamma ((\genk\Lconv \psit)_t, \psitt)_{L^2} \\
				& +\nu e^{\tilde \lambda t}(\tau (\genk\Lconv \nabla\psi_t)_t, \nabla \psi_t)_{L^2} =0.
			\end{aligned}  
		\end{equation}
		We use the fact that $\psi_1=0$ to write
		\[
		((\genk\Lconv \psit)_t, \psitt)_{L^2} = (\genk\Lconv \psitt, \psitt)_{L^2}.
		\]
		Hence, integrating in time and using \ref{assu4}--\ref{assu5} yields the desired result. 
	\end{proof}

	\begin{proof}[Proof of Proposition \ref{thm:decay_all_quants}]
		We intend to show that Lemma~\ref{lem:main_est_2} combined with Lemmas~\ref{lem:main_est}--\ref{lem:F2_est} yields exponential decay of $E[\psi]+{E_\textup{mod}}[\psi]$ and thus the exponential decay of $||\psi(t)||_{\rm{mod}}$, by exploiting Lemma \ref{Equiv_Lemma}.  To this end, let us define the modified Lyapunov functional $\mathcal{\tilde L}=\mathcal{\tilde L}(t)$ as 
		\begin{equation}\label{Lyapunov_2}
			\mathcal{\tilde L}(t)= N_0 \big(E[\psi](t)+{E_\textup{mod}} E[\psi](t)\big)+F_1(t)+ N_1F_2(t),
		\end{equation} 
		where $N_0$ and $N_1$ are large positive constants that will be fixed later.  
		
		Similarly to \eqref{Equiv_E_L}, it is clear that for $N_0$ large enough, for all $t\geq 0$, it holds that 
		\begin{equation}\label{Equiv_E_L_2}
			\mathcal{\tilde L}(t)\sim  E[\psi](t) +{E_\textup{mod}}{E}[\psi](t)
			\sim ||\psi(t)||_{\rm{mod}}^2
		\end{equation}
		Using \eqref{Lyapunov_2} together with the estimates \eqref{E_Main_Estimate},  \eqref{F_1_Main_Est_2} \eqref{F_2_Main}, and \eqref{E_Main_Estimate_2}, we obtain 
		\begin{equation}  
			\begin{aligned}\label{eq:intermediary_2}
				&e^{\tilde \lambda t}\mathcal{\tilde L}(t)+\Big(N_0(\gamma-\tau c^2)\tilde{c}- C(\varepsilon_1)-N_1C(\varepsilon_3)\Big)\intt e^{\tilde \lambda s} \|\genk\Lconv \nabla \psi_t\|_{L^2(\Om)}^2\ds\\
				&
				+N_0\tau\frac{\gamma-\tau c^2}\gamma\tilde{c}\int_0^t e^{\tilde \lambda s}\|\genk\Lconv \psitt\|_{L^2(\Om)}^2\ds\\&+ \Big(N_0\nu\Big(1+\frac{\tau c^2} \gamma\Big)- C(\varepsilon_1)- N_1C(\varepsilon_2,\varepsilon_3)\Big) \intt e^{\tilde \lambda s}\|\nabla \psi_t\|_{L^2(\Om)}^2\ds\\
				&+(c^2-\varepsilon_1-N_1\varepsilon_3)\intt e^{\tilde \lambda s} \|\nabla(\tau \genk\Lconv \psi_t+\psi)\|_{L^2(\Om)}^2\ds\\
				&+\Big(N_1(1-\varepsilon_2)-1\Big)\intt e^{ \tilde \lambda s} \|\tau (\genk\Lconv \psi_t)_t+\psi_t\|_{L^2(\Om)}^2\ds \\
				&\qquad\qquad\qquad\lesssim \big(E[\psi](0)+{E_\textup{mod}}[\psi](0)\big) + \tilde \lambda \intt e^{\tilde \lambda s} \big(E[\psi](s)+{E_\textup{mod}}[\psi](s)\big)\ds,
			\end{aligned}
		\end{equation}
		for all $0<\tilde \lambda\leq \lambda_{sup}$, where $\tilde c$ comes from \ref{assu4}.
		We fix $\varepsilon_1$ and $\varepsilon_2$ small enough such that $\varepsilon_1<c^2$ and $\varepsilon_2<1$. Once, $\varepsilon_2$ is fixed, we select $N_1$ large enough such that 
		\begin{equation}
			N_1(1-\varepsilon_2)>1. 
		\end{equation}
		After that, we take $\varepsilon_3$ small enough such that
		\begin{equation}
			\varepsilon_3<\frac{c^2-\varepsilon_1}{N_1}. 
		\end{equation}
		Once, all the above constants are fixed, we take $N_0$ large enough such that the coefficients in the first two integral terms are positive.  
		To finish, we pick 
		\begin{equation}
			\begin{multlined}
				C^* \tilde \lambda = \min\Big\{N_0(\gamma-\tau c^2)\tilde{c}- C(\epsilon_1)-N_1C(\varepsilon_3),N_0\tau\frac{\gamma-\tau c^2}\gamma\tilde{c},N_0\nu\Big(1+\frac{\tau c^2} \gamma\Big)-\\ C(\varepsilon_1)- N_1C(\varepsilon_2,\varepsilon_3),
				c^2-\varepsilon_1-N_1\varepsilon_3,N_1(1-\varepsilon_2)-1,C^*\lambda_{sup}\Big\},
			\end{multlined}
		\end{equation}
		where $C^*$ is the hidden constant in \eqref{eq:intermediary} once the constants $\varepsilon_{1,2,3}>0$ and $N_{0,1}>0$ have been fixed.
		Hence, by recalling \eqref{Equiv_E_L_2}, we obtain the desired result. 
	\end{proof}
	
	\section{A decay result under weaker conditions on the kernel}\label{sec:weak_decay}
	In Section~\ref{sec:asympto_behav_1}, {Set of Assumptions~}\ref{Assumption_Kernel_Exponential_Decay} was rather restrictive. It was specifically designed to provide an exponential control of certain unknowns. It is possible to state a weaker version of this assumption that allow us to show a weaker decay result of the energy $\|\psi\|_{\rm{mod}}$ defined in \eqref{def:psi_mod}. This result is given in Proposition~\ref{prop:decay_Lyap_weak} below. The assumption needed is given here. 
	
	\begin{assumption}[Weak control of the energy]\label{Kernel_Properties_Weak_Decay}
		We assume that
		\begin{enumerate}
			\item[($\mathcal{A}4_\textup{weak}$) \namedlabel{assu4weak}{($\mathcal{A}4_\textup{weak}$)}] Let there exist a constant $\tilde{c}_0>0$ such that
			\begin{equation}
				\intt  (\genk\Lconv y, y)_{L^2}\ds\geq \tilde{c}_0 \intt  \|\genk\Lconv y\|_{L^2(\Om)}^2\ds,
			\end{equation}
			for all $t\in \R^+$ and $y\in L^2(0,t;L^2(\Om))$.
			
			\item[($\mathcal{A}5_\textup{weak}$) \namedlabel{assu5weak}{($\mathcal{A}5_\textup{weak}$)}]
			\begin{equation}
				\intt (\genk\Lconv y_t, y)_{L^2}\ds\geq0
			\end{equation}
			for all $t\in \R^+$ and $y\in H^1\cap X_\genk(0,t;L^2(\Om))$ such that $y(0) = 0$.
		\end{enumerate}
	\end{assumption}
	Then, we can state the following result related to the decay of the energy $\|\psi\|_{\rm{mod}}.$
	\begin{proposition}[Weak energy decay]\label{prop:decay_Lyap_weak}
		Suppose that Assumptions~\ref{Assumption_Kernel} and~\ref{Kernel_Properties_Weak_Decay} on the kernel hold. Assume in addition that \eqref{ineqs:condition_coefficients} holds and that initial data are chosen as in Theorem~\ref{thm:wellp}, and that $\psi_1=0$. Then,
		\[ \|\psi(t)\|_{\rm{mod}} \to 0 \quad \textrm{as}\quad t\to\infty.\]
	\end{proposition}
	The main advantage of {Set of Assumptions~}\ref{Kernel_Properties_Weak_Decay} (compared to {Set of Assumptions~}\ref{Assumption_Kernel_Exponential_Decay}) is that we do not have to rely on the exponential decay of the kernel to verify it. Indeed, it suffices that $\genk\in L^1(\R^+)$ be positive, nonincreasing and convex for it to verify the assumptions of \cite[Theorem 2 (iii)]{nohel1976frequency} and \cite[Lemma 5.1]{kaltenbacher2022limiting},  thus satisfying {Set of Assumptions~}\ref{Kernel_Properties_Weak_Decay}. For example, the polynomially decaying kernel \eqref{Polynomial_Kernel}
	and the Mittag-Leffler function:
	\begin{equation}
		\genk=\frac{t^{\alpha-1}}{\Gamma(1-\alpha)}E_{\alpha,\alpha}(-t^\alpha)\quad\textrm{with } 0<\alpha\leq 1,
	\end{equation}
	verify {Set of Assumptions~}\ref{Kernel_Properties_Weak_Decay} but not the stronger {Set of Assumptions~}\ref{Assumption_Kernel_Exponential_Decay}. 
	
	\begin{proof}[Proof of Proposition~\ref{prop:decay_Lyap_weak}]
		{Set of Assumptions~}\ref{Kernel_Properties_Weak_Decay} corresponds to allowing $\lambda_{sup} = 0$ in {Set of Assumptions~}\ref{Assumption_Kernel_Exponential_Decay}. Thus, setting $\tilde\lambda=0$ in~\eqref{eq:intermediary_2}, yields
		\begin{equation}  
			\begin{aligned}
				&\mathcal{\tilde L}(t)+\Big(N_0(\gamma-\tau c^2)\tilde{c}- C(\varepsilon_1)-N_1C(\varepsilon_3)\Big)\intt  \|\genk\Lconv \nabla \psi_t\|_{L^2(\Om)}^2\ds\\
				&
				+N_0\tau\frac{\gamma-\tau c^2}\gamma\tilde{c}\int_0^t \|\genk\Lconv \psitt\|_{L^2(\Om)}^2\ds\\&+ \Big(N_0\nu\Big(1+\frac{\tau c^2} \gamma\Big)- C(\varepsilon_1)- N_1C(\varepsilon_2,\varepsilon_3)\Big) \intt \|\nabla \psi_t\|_{L^2(\Om)}^2\ds\\
				&+(c^2-\varepsilon_1-N_1\varepsilon_3)\intt \|\nabla(\tau \genk\Lconv \psi_t+\psi)\|_{L^2(\Om)}^2\ds\\
				&+\Big(N_1(1-\varepsilon_2)-1\Big)\intt \|\tau (\genk\Lconv \psi_t)_t+\psi_t\|_{L^2(\Om)}^2\ds \lesssim \big(E[\psi](0)+{E_\textup{mod}}[\psi](0)\big).
			\end{aligned}
		\end{equation}
		where after choosing $\varepsilon_{1,2,3}>0$ similarly to the proof of Proposition~\ref{thm:decay_all_quants} one obtains, after letting $t\to\infty$, that
		$$\int^\infty_0  \|\psi(s)\|_{\rm{mod}}^2\ds < \infty.$$
		Thus $ \|\psi(t)\|_{\rm{mod}}\to 0$ as $t\to\infty$.

	\end{proof}
	
	Note that Proposition~\ref{prop:decay_Lyap_weak} gives no explicit decay rate.
	This result is most likely not sharp in the sense that, in the case of viscoelastic wave equation with nonlocal dissipation with memory kernel~\eqref{Polynomial_Kernel}, one expects a decay rate depending on $p$, see, e.g., \cite{munoz1996decay}.
	
	\section{Conclusion}
	In this article, we have established a new global well-posedness result for a generalized Moore--Gibson--Thompson equation by leaning on the strong damping ($\nu \D \psi_t$) to eliminate the need of the existence of a regular enough resolvent for the kernel $\genk$. Moreover, we have shown, through the construction of a suitable Lyapunov functional, the exponential decay of the energy which holds provided that the memory kernel satisfies certain assumptions and have given examples of such kernels. We have also shown that under weaker assumptions, one can still obtain that the energy of the solution vanishes. 
	
	Future work will focus on establishing optimality of the decay rates for the nonlocal Moore--Gibson--Thompson equations covered by the current work. We expect that this question can be treated using, e.g., the complete description of the spectrum of the Laplacian combined with other techniques along the lines of~\cite{pellicer2019wellposedness,2019optimal}.

	While the well-posedness theory presented here covers the case of fractional kernels, the limiting behavior framework does not satisfactorily address their long-term behavior. The main challenge resides in being able to use the fractional derivative terms to argue that one can control part of the energy with it (see, e.g., \ref{assu4} and \ref{assu4weak}). One may, along the lines of, e.g.,~\cite{fritz2022equivalence,lasiecka2017global}, attempt to consider different definitions of the energy of the system which might be easier to treat. The right choice of such an energy functional for \eqref{eq:first} remains open.
	
\section*{Acknowledgements}
The authors thank the reviewer for his/her careful reading of the manuscript and helpful remarks, which have led to marked improvements.
They are also grateful to Dr. Vanja Nikoli\'c (Radboud University) for her careful reading and extensive feedback on drafts of the manuscript and for many thought-provoking discussions.
\section{Compactness and embeddings of the space of solutions}\label{Sec:Preliminaries}
Towards proving Theorem~\ref{thm:wellp}, we needed certain compactness properties on the spaces of solutions which is needed to pass to the limit in the Galerkin procedure. In what follows we provide the Lemmas needed for the limiting procedure as well as their proofs.

First, we consider the space $H^1\cap X_\genk^\infty(0,T)$ which can be thought of as an augmented version of the $X_\genk^p(0,T)$ space introduced in \cite{meliani2023unified}. This augmentation is shown to give nice completeness properties regardless of the regularity of the resolvent $\tgenk$.
\begin{lemma}[Compactness of $H^1 \cap X_\genk^\infty(0,T)$ defined in \eqref{def_Xfp}] \label{Lemma:hybrid_Caputo_seq_compact}
	Let $\genk \in \mm(0,T)$, then $H^1 \cap X_\genk^\infty(0,T)$ is a Banach space.
	Furthermore, the unit ball of $H^1 \cap X_\genk^\infty(0,T)$, $B_\genk^\infty$, is weak-$*$ sequentially compact. Naturally, when $T$ is finite we have the embeddings $$H^1 \cap X_\genk^\infty(0,T)\hookrightarrow H^1(0,T)\hookrightarrow C[0,T].$$
\end{lemma}
\begin{proof} 
	Completeness and sequential compactness of the space $H^1 \cap X_\genk^\infty$ follow from a slight modification of the arguments in \cite[Lemma 1]{meliani2023unified} by taking advantage of the embedding $H^1(0,T) \hookrightarrow C[0,T]$. We omit the details.
\end{proof}

In the course of the proof of Theorem~\ref{thm:wellp}, we needed to argue that by proving a uniform-in-$n$ (where $n$ is the Galerkin discretization parameter) bound on only the summed quantity $\|(\tau\genk\Lconv\psi_{t}+\psi)_{tt}\|_{L^2(0,T;H^{-1}(\Om))}$, we can still pass to the limit in the Galerkin procedure. To this end, we use the following lemma. 

\begin{lemma}\label{lem:sq_compactness_highest_der}
	Let $\genk \in \mm_{\textup{loc}}(\R^+)$. Then the space $V_{\genk}(0,T)$ defined in \eqref{V_R_Space} and endowed with the norm \eqref{V_R_Norm} is complete, and its unit ball is weak sequentially compact.
\end{lemma}
\begin{proof}
	To show that $V_{\genk}(0,T)$ is complete, take a Cauchy sequence $(u_n)_{n\geq1}\subset V_{\genk}(0,T)$. Then $(u_n)_{n\geq1}$, $(u_{nt})_{n\geq1}$, and $(\genk\Lconv u_{nt}+u_n)_{tt} \in L^2(0,T)$ are Cauchy sequences in $L^2(0,T)$ and therefore converge to some limits $u$, $g$, and $h$ respectively, in $L^2(0,T)$. 
	
	It is easy to see that necessarily $g = u_t$ weakly, see e.g.,~\cite[Section 8.2]{brezis2010functional} for details on the (completeness) properties for Sobolev spaces.
	
	Towards showing that $h = (\genk\Lconv u_{t}+u)_{tt}$, let us first notice that since $\genk \in \mm_{\textup{loc}}(\R^+)$, the operator
	\begin{equation}
		\begin{aligned}
			\operatorname{\mathcal{T}}_{\genk}\ : L^2(0,T) &\rightarrow L^2(0,T) 
			\\
			v&\mapsto \genk\Lconv v
		\end{aligned} 
	\end{equation}
	is continuous and therefore $\genk\Lconv u_{nt} \to \genk\Lconv u_{t}$ in $L^2(0,T)$. Let now $n\geq 1$ and take an arbitrary $\phi \in C_c^2([0,T])$. Integrating by parts twice yields
	\[\intT (\genk\Lconv u_{nt}+u_n)_{tt} \, \phi \ds = \intT (\genk\Lconv u_{nt}+u_n) \,\phi_{tt} \ds.\]
	Passing to the limit, we obtain 
	\[\intT h \,\phi \ds = \intT (\genk\Lconv u_{t}+u) \, \phi_{tt} \ds.\]
	Thus, $u \in V_{\genk}(0,T)$, $ h = (\genk\Lconv u_{t}+u)_{tt}$ weakly, and $\|u_n-u\|_{V_{\genk}(0,T)} \to 0$ as $n\to \infty$.
	
	The compactness of the unit ball follows similarly to, e.g., \cite[Lemma 1]{meliani2023unified} by relying on the properties of $L^p$ spaces.
\end{proof}

We provide an embedding lemma for the space $V_{\genk}(0,T)$. This result is helpful in showing that initial data are attained in the course of the proof of Theorem~\ref{thm:wellp}.

\begin{lemma}\label{lem:continuous_embedding}
	Let $\genk \in \mm_{\textup{loc}}(\R^+)$. Assume that ($\mathcal{A}3$) holds.
	Then we have the continuous embedding   
	$$V_\genk(0,T) \hookrightarrow \{ v \in W^{2,p}(0,T)\,|\, \genk\Lconv v_t \in W^{2,p}(0,T)\},$$
	with $p = \min\{2,q\}>1$.
\end{lemma}
\begin{proof}
	Let $u\in V_\genk(0,T)$. Then  $ (\genk\Lconv u_{t})_{tt} + u_{tt} \in L^2(0,T)$.
	Consider then the auxiliary problem
	\begin{equation}\label{phi_equality} 
		\tau(\genk\Lconv u_{t})_{tt} + u_{tt} = \phi\quad\textrm{a.e. on} \quad (0,T)
	\end{equation}
	with $\phi \in L^2(0,T)$. After convolving with the resolvent $\tgenk$, this yields
	\begin{equation}\label{eq:bootstrap}
		(\tau+A) u_{tt}+ {\mathsf{r}} \Lconv u_{tt} = f \in L^p(0,T),
	\end{equation}
	with $p = \min\{2,q\}>1$ and, when $\|\genk\|_{\mm(\{0\})} = 0$, the term $f$ is given by
	$$f = A \phi +  {\mathsf{r}} \Lconv \phi + (\tau \genk\Lconv u_{t})_t(0) \, \mathsf{r}.$$
	In general, $f$ can be expressed as 
	$$f = A \phi +  {\mathsf{r}} \Lconv \phi + (\tau \genk\Lconv u_{t})_t(0) \, \mathsf{r} + (\tau \genk\Lconv u_{t})(0) \mathsf{r}_t,$$ hence the assumption that $\mathsf{r}$ should be in $W^{1,q}(0,T)$ when $\|\genk\|_{\mm(\{0\})} \neq 0$.
	
	Existence theory for Volterra equations of the second kind (see~\cite[Ch.\ 2, Theorem 3.5]{gripenberg1990volterra}) yields then that $$ u_{tt} \in L^p(0,T) \quad \textrm{implying that}\quad u \in W^{2,p}(0,T)$$ and in turn $\genk\Lconv u_t\in W^{2,p}(0,T)$.
	
	The continuous dependence of $u_{tt}$ on $f$ in the topology of $L^p(0,T)$ is guaranteed by the variation of constants formula~\cite[Ch.\ 2, Equation (3.2)]{gripenberg1990volterra}.
	Thus the continuous embedding predicted in the statement of the lemma is verified if we can show that 
	\[\|f\|_{L^p(0,T)} \leq C_T \|u\|_{V_\genk(0,T)},\]
	with $C_T$ being a generic constant that depends on $T$.
	To this end, let us slightly rewrite the right-hand side of \eqref{eq:bootstrap} as
	\begin{equation}\label{eq:embedding}
		f = A \phi +  {\mathsf{r}} \Lconv \phi + (\tau \genk\Lconv u_{t}+u)_t(0) \, \mathsf{r} - u_t(0) \mathsf{r}.
	\end{equation}
	First, notice that $\|\phi\|_{L^p} \leq C_T \|u\|_{V_\genk(0,T)}$ from \eqref{phi_equality}. Furthermore, because $\genk\Lconv u_{t} + u \in H^2(0,T)$, we have that 
	$$
	|(\genk\Lconv u_{t})_t(0) + u_t(0)|  \leq\, C_T \|u\|_{V_\genk(0,T)}
	$$
	We need to establish the continuous dependence of $|u_t(0)|$ on $\|u\|_{V_\genk(0,T)}$. To this end, let us start with the auxiliary equation
	\begin{equation} \label{eq:bootstrap_lower}
		\tau(\genk\Lconv u_{t})_{t} + u_{t} = \Phi\quad\textrm{a.e. on} \quad (0,T),
	\end{equation}
	with $\Phi \in H^1(0,T)$ (since $u \in V_{\genk}(0,T)$). By setting $\chi = \genk\Lconv u_{t}$, we can rewrite \eqref{eq:bootstrap_lower} as: 
	\begin{equation} 
		(\tau+A)\chi_t + \mathsf{r}\Lconv \chi_{t} = \Phi\quad\textrm{a.e. on} \quad (0,T),
	\end{equation}
	Volterra existence theory then yields that $\|\chi\|_{H^1(0,T)} = \|\genk \Lconv u_t\|_{H^1(0,T)}$ continuously depends on $\|u\|_{V_\genk(0,T)}$ (via the variation of constants formula).
	On the other hand, let $t \in (0,T)$, we can express:
	\begin{equation}\label{eq:smaller_embedding}
		u_t(t) = \tgenk \Lconv (\genk\Lconv u_t)_t + \mathsf{r} (\genk\Lconv u_t)(0) +u_t(0).
	\end{equation}
	Then, similarly to the proof of the Sobolev embedding~\cite[Theorem 4.12]{adams2003sobolev} where, instead of the Taylor expansion~\cite[Lemma 4.15]{adams2003sobolev}, we use the relation~\eqref{eq:smaller_embedding} as starting point, we obtain the continuous dependence of $\|u_t\|_{C(0,T)}$ on $\|u\|_{V_\genk(0,T)}$. Going back to \eqref{eq:embedding}, we obtain that indeed
	\[\|f\|_{L^p(0,T)} \leq C_T \|u\|_{V_\genk(0,T)},\]
	proving the desired result.
\end{proof}

\section{Unique solvability of the semi-discrete problem}\label{appendix:solvability_semi_discrete}
For completeness, we provide here details on the Galerkin procedure used to show well-posedness of \eqref{Main_System}. Given an orthogonal basis $\{\phi_n\}_{n \geq 1}$ of $V=H_0^1(\Om)$, let $V_n=\text{span}\{\phi_1, \ldots, \phi_n\} \subset V$ and
\begin{equation}
	\begin{aligned}
		\psi^{(n)}(t) = \sum_{i=1}^n \xin (t) \phi_i.
	\end{aligned}	
\end{equation}
Choose the approximate initial data
\begin{equation}
	\psi^{(n)}_0=  \sum_{i=1}^n \xi_i^{(0, n)} \phi_i,\quad \psi^{(n)}_1= \sum_{i=1}^n \xi_i^{(1, n)} \phi_i, \quad  \psi^{\genk,\,(n)}_2 = \sum_{i=1}^n \xi_i^{(2, n)} \phi_i \in V_n,
\end{equation}
such that
\begin{equation} \label{convergence_approx_initial_data_z_form}
	\begin{aligned}
		\psi^{(n)}_0 \rightarrow \psi_0 \ \textrm{ and } \psi^{(n)}_1 \rightarrow \psi_1 \ \text{in} \ H_0^1(\Om), \ \text{and } \, \psi^{\genk, (n)}_2 \rightarrow \psitwo \  \text{in} \ L^2(\Om), \textrm{ as } \ n \rightarrow \infty.
	\end{aligned}
\end{equation}
For each $n \in \N$, the system of Galerkin equations is given by
\begin{equation}
	\begin{aligned}
		\begin{multlined}[t]\tau	\sum_{i=1}^n (\genk * \xit)_{tt}(t) (\phi_i, \phi_j)_{L^2}+ \sum_{i=1}^n \xitt (\phi_i, \phi_j)_{L^2}+c^2  \sum_{i=1}^n \xin_i ( \nabla \phi_i, \nabla\phi_j)_{L^2} \\
			+\gamma \sum_{i=1}^n (\genk * \xit)(t) (\nabla \phi_i, \nabla \phi_j)_{L^2} + \nu \sum_{i=1}^n \xit(\nabla \phi_i, \nabla \phi_j)_{L^2} =0
		\end{multlined}	
	\end{aligned}		
\end{equation}
for a.e.\ $t \in (0,T)$ and all $j \in \{1, \ldots, n\}$. With $\bxi = [\xin_1 \ \ldots \  \xin_n]^T$, we can write this system in matrix form 
\begin{equation}\label{equation:system_appendix}
	\left \{	\begin{aligned}
		& \tau M	(\genk * \bxitt)_t + M \bxitt+c^2 K \bxi+\gamma K  \genk* \bxit + \nu K \bxit = 0, \\[1mm]
		& (\bxin, \bxit, \genk\Lconv\bxitt)\vert_{t=0} = (\boldsymbol{\xi_0}, \boldsymbol{\xi_1}, \boldsymbol{\xi_2}^{\genk}),
	\end{aligned} \right.
\end{equation}
where $({\bxi_0}, {\bxi_1},  \bxi_2^{\genk})=([\xi_1^{(0,n)} \, \ldots \, \xi_n^{(0, n)}]^T,\, [\xi_1^{(1,n)} \, \ldots \, \xi_n^{(1, n)}]^T,\,[\xi_1^{(2,n)} \, \ldots \, \xi_n^{(2, n)}]^T)$. Here, $M$ and $K$ are the standard mass and stiffness matrices whose entries are given by:
\[ M_{i,j} = \int_\Omega \phi_i \ \phi_j \dx, \qquad  K_{i,j} = \int_\Omega \nabla \phi_i \cdot \nabla \phi_j \dx, \]
for $1 \leq i,j \leq n$.
To prove that the Galerkin system is uniquely solvable, we introduce the new unknown \[\hat{\boldsymbol\chi} = (\genk\Lconv\boldsymbol{\xi}_{t})_{tt}.\]

Thus, recalling the resolvent $\tgenk=A\delta_0 + \mathsf{r}$, we can further
express
\begin{equation}
	\begin{aligned}
		(\genk\Lconv\boldsymbol{\xi}_{t})_{t} &= 	1\Lconv \hat{\boldsymbol\chi} +   \bxi_2^{\genk} \\
		\bxi_{t} &= 1\Lconv \tgenk \Lconv \hat{\boldsymbol\chi} + \bxi_2^{\genk} \Lconv \tgenk + (\genk\Lconv \bxi_t)(0)\, \mathsf{r}
		\\
		\bxi &= 1\Lconv 1 \Lconv \tgenk \Lconv \hat{\boldsymbol\chi} + \bxi_2^{\genk} \Lconv 1 \Lconv \tgenk + (\genk\Lconv \bxi_t)(0)\Lconv\mathsf{r} +\bxi_0.
	\end{aligned}
\end{equation}
Thus, we will only be able to express $\bxi_{tt}$ if the term $(\genk\Lconv \bxi_t)(0)\, \mathsf{r}$ in the expression of $\bxi_t$ is differentiable. Which we discuss next by treating two cases separately.

\subsection*{Case 1: kernel $\genk$ contains a point mass at 0 ($\|\genk\|_{\mm(\{0\})}\neq0$)}

Then, by Assumption~\ref{assu3}, we know that $\mathsf{r} \in W^{1,1}(0,T)$. This allows us to differentiate the relation of $\bxi_t$ to express $\bxitt$. In particular, the system of equations is written as
\begin{equation}\label{semidiscrete_pointmass_var}
	\begin{aligned}
		\boldsymbol\xi_{tt} =& \, {\tgenk}\Lconv \hat{\boldsymbol\chi} + \bxi_2^{\genk} \mathsf{r} +(\genk\Lconv \bxi_t)(0)\, \mathsf{r}_t ,
		\\
		\boldsymbol\xi_t =& \,1 \Lconv {\tgenk}\Lconv \hat{\boldsymbol\chi} + \bxi_2^{\genk} \Lconv\tgenk+ (\genk\Lconv \bxi_t)(0)\, \mathsf{r},
		\\
		\boldsymbol\xi =& \,1 \Lconv 1 \Lconv {\tgenk}\Lconv \hat{\boldsymbol\chi} + \bxi_2^{\genk} \Lconv 1\Lconv\tgenk+
		(\genk\Lconv \bxi_t)(0)\Lconv \mathsf{r} + \boldsymbol\xi_0.
	\end{aligned}
\end{equation}
With these relations, we can rewrite~\eqref{equation:system_appendix} into a Volterra equation of the second kind for $\hat{\boldsymbol\chi}$. To avoid unnecessary repetition, we will only do it below for the case where the kernel $\genk$ does not contain a point mass at 0 and omit it in the current case.

\subsection*{Case 2: kernel $\genk$ does not contain a point mass at 0 ($\|\genk\|_{\mm(\{0\})}=0$)}
In this case, we can ascertain that
$$(\genk\Lconv \bxi_t)(0) =0.$$
The equation of $\bxi_t$ above evaluated at $t=0$ indicates that the second initial datum has to verify the compatibility condition
$$\bxi_1= A\bxi_2^{\genk}.$$ 

Thus, when $\|\genk\|_{\mm(\{0\})}=0$, we can rewrite the semi-discrete fractional derivative system using
\begin{equation}\label{semidiscrete_var}
	\begin{aligned}
		\boldsymbol\xi_{tt} =& \, {\tgenk}\Lconv \hat{\boldsymbol\chi} + \bxi_2^{\genk} \mathsf{r} ,
		\\
		\boldsymbol\xi_t =& \,1 \Lconv {\tgenk}\Lconv \hat{\boldsymbol\chi} + \bxi_2^{\genk} \Lconv\mathsf{r}+ \bxi_1,
		\\
		\boldsymbol\xi =& \,1 \Lconv 1 \Lconv {\tgenk}\Lconv \hat{\boldsymbol\chi} + \bxi_2^{\genk} \Lconv 1\Lconv\mathsf{r} + 1 \Lconv \bxi_1 
		+ \boldsymbol\xi_0,
	\end{aligned}
\end{equation}
with $\bxi_1 = A\bxi_2^{\genk}$.
This yields the following equation for $\hat{\boldsymbol\chi}$:
\begin{equation}\label{wellposed_semidiscrete_rewriting_1}
	\begin{aligned}
		\begin{multlined}[t] \tau	\hat{\boldsymbol\chi}+ {\tgenk}* \hat{\boldsymbol\chi} + c^2 M^{-1}K \, 1*1*{\tgenk}*\hat{\boldsymbol\chi}
			+\gamma M^{-1}K \, 1* 1*\hat{\boldsymbol\chi}\\+  \nu M^{-1}K \, 1 *{\tgenk}* \hat{\boldsymbol\chi}  = {\boldsymbol{ f}},
		\end{multlined}
	\end{aligned}
\end{equation}
where the source term is
\begin{equation}
	\begin{aligned}
		{\boldsymbol{ f}}=\, 	\begin{multlined}[t] - \bxi_2^{\genk} \mathsf{r}- c^2M^{-1}K(\bxi_2^{\genk} 1\Lconv 1\Lconv\mathsf{r} + 1 \Lconv \bxi_1 + \boldsymbol\xi_0) -\gamma M^{-1} K(1\Lconv 1)\bxi_2^{\genk} \\- \nu M^{-1}K (\bxi_2^{\genk} 1\Lconv\mathsf{r}+ \bxi_1).	\end{multlined} 
	\end{aligned}
\end{equation}
Using that $\tgenk = A \delta_0 + \mathsf{r}$, we can further rewrite \eqref{wellposed_semidiscrete_rewriting_1} as
\begin{equation}
	\begin{aligned}
		\begin{multlined}[t] (\tau+A)	\hat{\boldsymbol\chi}+ \mathsf{r}* \hat{\boldsymbol\chi} + c^2 M^{-1}K \, 1*1*\mathsf{r}*\hat{\boldsymbol\chi}
			+\gamma M^{-1}K \, 1* 1*\hat{\boldsymbol\chi}\\+  \nu M^{-1}K \, 1 *\mathsf{r}* \hat{\boldsymbol\chi} + c^2 M^{-1}K \, 1*1*\hat{\boldsymbol\chi}
			+  \nu A M^{-1}K \, 1* \hat{\boldsymbol\chi}  = {\boldsymbol{ f}}.
		\end{multlined}
	\end{aligned}
\end{equation}

Since $\mathsf{r}\in L^1_\textup{loc}(\R^+)$ and $\boldsymbol{ f} \in L^1(0,T)$,
we obtain, through \cite[Ch. 2, Theorem 3.5]{gripenberg1990volterra}, that the system has a unique solution $\hat{\boldsymbol\chi} \in L^1(0,T)$.
In turn, we have that $\boldsymbol{\xi} \in \{\boldsymbol{u} \in W^{2,1}(0,T) \ |\ (\genk\Lconv\boldsymbol{u}_t)_{t} \in W^{1,1}(0,T)\}$.
Thus, $\psi^{(n)} \in \{u \in W^{2,1}(0,T;V_n) \ |\ (\genk\Lconv u_t)_t \in W^{1,1}(0,T;V_n)\}$.

\section{Existence of a regular resolvent}\label{App_sufficient_condition}
The analysis above relies on a number of assumptions among which {we find \ref{assu3}.} This assumption states that if $\|\genk\|_{\mm(\{0\})} \neq 0 $, we need that $\mathsf{r} \in W^{1,q}(0,T)$ for some $q>1$. We give here a lemma which ensures this regularity for a class of relevant examples.
\begin{lemma}
	Let the memory kernel be given by
	$$\genk = B \delta_0 + \mathfrak{s},$$
	with $B\in \R\setminus\{0\}$ and $\mathfrak{s} \in L^q(0,T)$ for $q\geq 1$. Then, $\|\genk\|_{\mm(\{0\})} = |B|$ and $\genk$ has a resolvent $\tgenk \in W^{1,q}(0,T) $.
\end{lemma}
\begin{proof}
	By writing the definition of a resolvent, we arrive at $\tgenk$ verifying
	\begin{equation}\label{Appendix:VIE2}
		B \tgenk + \mathfrak{s}\Lconv \tgenk =1.
	\end{equation}
	Since $\mathfrak{s} \in L^1(0,T)$ and $B \neq 0$, equation \eqref{Appendix:VIE2} is a Volterra integral equation of the second kind. Thus existence theory of such equations ensures the existence of a unique solution $\tgenk \in C[0,T]$; see, e.g., \cite[Ch.\ 2, Theorem 3.5]{gripenberg1990volterra}. From evaluating \eqref{Appendix:VIE2} at 0, it is clear that $\tgenk(0) = \dfrac1 B$.
	
	To prove that $\tgenk$  belongs to $W^{1,q}(0,T)$, we need to show that its (weak) derivative is in $L^q(0,T)$. Notice to this end that the derivative has to verify the following Volterra equation of the second kind:
	\begin{equation}\label{Appendix:VIE2_der}
		B \tgenk_t + \mathfrak{s}\Lconv \tgenk_t = - \frac{\mathfrak{s}}{B}.
	\end{equation}
	The same existence theory invoked above then ensures that $\tgenk_t \in L^q(0,T)$, which completes the proof.
\end{proof}


\begin{thebibliography}{10}
	
	\bibitem{adams2003sobolev}
	{\sc R.~A. Adams and J.~J. Fournier}, {\em Sobolev Spaces}, Elsevier, 2003.
	
	\bibitem{ammari2011mathematical}
	{\sc H.~Ammari}, {\em Mathematical Modeling in Biomedical Imaging II: Optical,
		Ultrasound, and Opto-Acoustic Tomographies}, vol.~2035, Springer Science \&
	Business Media, 2011.
	
	\bibitem{Bose_Gorai_1998}
	{\sc S.~K. Bose and G.~C. Gorain}, {\em Stability of the boundary stabilised
		internally damped wave equation {$y''+\lambda y'''=c^2(\Delta y+\mu\Delta
			y')$} in a bounded domain in {${\bf R}^n$}}, Indian J. Math., 40 (1998),
	pp.~1--15.
	
	\bibitem{brezis2010functional}
	{\sc H.~Brezis}, {\em Functional Analysis, Sobolev Spaces and Partial
		Differential Equations}, Springer Science \& Business Media, 2010.
	
	\bibitem{bucci2020regularity}
	{\sc F.~Bucci and L.~Pandolfi}, {\em On the regularity of solutions to the
		{M}oore--{G}ibson--{T}hompson equation: a perspective via wave equations with
		memory}, Journal of Evolution Equations, 20 (2020), pp.~837--867.
	
	\bibitem{Chen_Ikehata_2021}
	{\sc W.~Chen and R.~Ikehata}, {\em The {C}auchy problem for the
		{M}oore-{G}ibson-{T}hompson equation in the dissipative case}, J.
	Differential Equations, 292 (2021), pp.~176--219.
	
	\bibitem{compte1997generalized}
	{\sc A.~Compte and R.~Metzler}, {\em The generalized {C}attaneo equation for
		the description of anomalous transport processes}, Journal of Physics A:
	Mathematical and General, 30 (1997), p.~7277.
	
	\bibitem{conti2007decay}
	{\sc M.~Conti, S.~Gatti, and V.~Pata}, {\em Decay rates of {V}olterra equations
		on {$\mathbb{R}^N$}}, Open Mathematics, 5 (2007), pp.~720--732.
	
	\bibitem{conti2021Moore}
	{\sc M.~Conti, L.~Liverani, and V.~Pata}, {\em On the
		{M}oore-{G}ibson-{T}hompson equation with memory with nonconvex kernels},
	Indiana University Mathematics Journal,  (2021).
	
	\bibitem{conti2020thermoelasticity}
	{\sc M.~Conti, V.~Pata, and R.~Quintanilla}, {\em Thermoelasticity of
		moore--gibson--thompson type with history dependence in the temperature},
	Asymptotic Analysis, 120 (2020), pp.~1--21.
	
	\bibitem{dell2017moore}
	{\sc F.~Dell'Oro and V.~Pata}, {\em On the {M}oore--{G}ibson--{T}hompson
		equation and its relation to linear viscoelasticity}, Applied Mathematics \&
	Optimization, 76 (2017), pp.~641--655.
	
	\bibitem{evans2010partial}
	{\sc L.~C. Evans}, {\em Partial Differential Equations}, vol.~2, Graduate
	Studies in Mathematics, AMS, 2010.
	
	\bibitem{fritz2022equivalence}
	{\sc M.~Fritz, U.~Khristenko, and B.~Wohlmuth}, {\em Equivalence between a
		time-fractional and an integer-order gradient flow: The memory effect
		reflected in the energy}, Advances in Nonlinear Analysis, 12 (2022),
	p.~20220262.
	
	\bibitem{Gorain_Bose_1998}
	{\sc G.~C. Gorain and S.~K. Bose}, {\em Exact controllability and boundary
		stabilization of torsional vibrations of an internally damped flexible space
		structure}, J. Optim. Theory Appl., 99 (1998), pp.~423--442.
	
	\bibitem{gripenberg1990volterra}
	{\sc G.~Gripenberg, S.-O. Londen, and O.~Staffans}, {\em Volterra Integral and
		Functional Equations}, no.~34, Cambridge University Press, 1990.
	
	\bibitem{holm2019waves}
	{\sc S.~Holm}, {\em Waves with Power-Law Attenuation}, Springer, 2019.
	
	\bibitem{holm2023adding}
	{\sc S.~Holm, S.~N. Chandrasekaran, and S.~P. N{\"a}sholm}, {\em Adding a low
		frequency limit to fractional wave propagation models}, Frontiers in Physics,
	11 (2023).
	
	\bibitem{holm2011causal}
	{\sc S.~Holm and S.~P. N{\"a}sholm}, {\em A causal and fractional all-frequency
		wave equation for lossy media}, The Journal of the Acoustical Society of
	America, 130 (2011), pp.~2195--2202.
	
	\bibitem{hrusa1988model}
	{\sc W.~J. Hrusa and M.~Renardy}, {\em A model equation for viscoelasticity
		with a strongly singular kernel}, SIAM Journal on mathematical Analysis, 19
	(1988), pp.~257--269.
	
	\bibitem{kaltenbacher2011wellposedness}
	{\sc B.~Kaltenbacher, I.~Lasiecka, and R.~Marchand}, {\em Wellposedness and
		exponential decay rates for the {M}oore--{G}ibson--{T}hompson equation
		arising in high intensity ultrasound}, Control and Cybernetics, 40 (2011),
	pp.~971--988.
	
	\bibitem{KatLasPos_2012}
	{\sc B.~Kaltenbacher, I.~Lasiecka, and M.~K. Pospieszalska}, {\em
		Well-posedness and exponential decay of the energy in the nonlinear
		{J}ordan-{M}oore-{G}ibson-{T}hompson equation arising in high intensity
		ultrasound}, Math. Models Methods Appl. Sci., 22 (2012), pp.~1250035, 34.
	
	\bibitem{kaltenbacher2022limiting}
	{\sc B.~Kaltenbacher, M.~Meliani, and V.~Nikoli{\'c}}, {\em Limiting behavior
		of quasilinear wave equations with fractional-type dissipation}, Advanced
	Nonlinear Studies,  (2024).
	
	\bibitem{kaltenbacher2020vanishing}
	{\sc B.~Kaltenbacher and V.~Nikoli\'c}, {\em Vanishing relaxation time limit of
		the {J}ordan--{M}oore--{G}ibson--{T}hompson wave equation with {N}eumann and
		absorbing boundary conditions}, Pure and Applied Functional Analysis, 5
	(2020), pp.~1--26.
	
	\bibitem{kaltenbacher2022time}
	{\sc B.~Kaltenbacher and V.~Nikoli{\'c}}, {\em Time-fractional
		{M}oore--{G}ibson--{T}hompson equations}, Mathematical Models and Methods in
	Applied Sciences, 32 (2022), pp.~965--1013.
	
	\bibitem{kaltenbacher2023vanishing}
	\leavevmode\vrule height 2pt depth -1.6pt width 23pt, {\em The vanishing
		relaxation time behavior of multi-term nonlocal
		jordan--moore--gibson--thompson equations}, Nonlinear Analysis: Real World
	Applications, 76 (2024), p.~103991.
	
	\bibitem{lasiecka2017global}
	{\sc I.~Lasiecka}, {\em Global solvability of {M}oore--{G}ibson--{T}hompson
		equation with memory arising in nonlinear acoustics}, Journal of Evolution
	Equations, 17 (2017), pp.~411--441.
	
	\bibitem{lasiecka2015moore}
	{\sc I.~Lasiecka and X.~Wang}, {\em Moore--{G}ibson--{T}hompson equation with
		memory, part {II}: {G}eneral decay of energy}, Journal of Differential
	Equations, 259 (2015), pp.~7610--7635.
	
	\bibitem{Trigg_et_al}
	{\sc R.~Marchand, T.~McDevitt, and R.~Triggiani}, {\em An abstract semigroup
		approach to the third-order {M}oore--{G}ibson--{T}hompson partial
		differential equation arising in high-intensity ultrasound: structural
		decomposition, spectral analysis, exponential stability}, Math. Methods.
	Appl. Sci., 35 (2012), pp.~1896--1929.
	
	\bibitem{meliani2023unified}
	{\sc M.~Meliani}, {\em A unified analysis framework for generalized fractional
		moore--gibson--thompson equations: Well-posedness and singular limits},
	Fractional Calculus and Applied Analysis, 26 (2023), pp.~2540--2579.
	
	\bibitem{messaoudi2007global}
	{\sc S.~A. Messaoudi, B.~Said-Houari, and N.-E. Tatar}, {\em Global existence
		and asymptotic behavior for a fractional differential equation}, Applied
	Mathematics and Computation, 188 (2007), pp.~1955--1962.
	
	\bibitem{moore1960propagation}
	{\sc F.~Moore and W.~Gibson}, {\em Propagation of weak disturbances in a gas
		subject to relaxation effects}, Journal of the Aerospace Sciences, 27 (1960),
	pp.~117--127.
	
	\bibitem{munoz1996decay}
	{\sc J.~E. Mu{\~n}oz~Rivera and E.~C. Lapa}, {\em Decay rates of solutions of
		an anisotropic inhomogeneous n-dimensional viscoelastic equation with
		polynomially decaying kernels}, Communications in Mathematical Physics, 177
	(1996), pp.~583--602.
	
	\bibitem{nikolic2023nonlinear}
	{\sc V.~Nikoli{\'c}}, {\em Nonlinear acoustic equations of fractional higher
		order at the singular limit}, Nonlinear Differential Equations and
	Applications NoDEA, 31 (2024), p.~30.
	
	\bibitem{nohel1976frequency}
	{\sc J.~Nohel and D.~Shea}, {\em Frequency domain methods for {V}olterra
		equations}, Advances in Mathematics, 22 (1976), pp.~278--304.
	
	\bibitem{parker2022power}
	{\sc K.~J. Parker}, {\em Power laws prevail in medical ultrasound}, Physics in
	Medicine \& Biology, 67 (2022).
	
	\bibitem{pellicer2019wellposedness}
	{\sc M.~Pellicer and B.~Said-Houari}, {\em Wellposedness and decay rates for
		the {C}auchy problem of the {M}oore--{G}ibson--{T}hompson equation arising in
		high intensity ultrasound}, Applied Mathematics \& Optimization, 80 (2019),
	pp.~447--478.
	
	\bibitem{Pellicer:2021aa}
	{\sc M.~Pellicer and B.~Said-Houari}, {\em On the {C}auchy problem of the
		standard linear solid model with {F}ourier heat conduction}, Zeitschrift
	f{\"u}r Angewandte Mathematik und Physik, 72, 115 (2021).
	
	\bibitem{P-SM-2019}
	{\sc M.~Pellicer and J.~Sol{\`a}-Morales.}, {\em Optimal scalar products in the
		{M}oore--{G}ibson--{T}hompson equation}, Evolution Equations and Control
	Theory, 8 (2019), pp.~203--220.
	
	\bibitem{2019optimal}
	{\sc M.~Pellicer and J.~Sol{\`a}-Morales}, {\em Optimal scalar products in the
		{M}oore--{G}ibson--{T}hompson equation}, Evolution Equations \& Control
	Theory, 8 (2019), p.~203.
	
	\bibitem{quintanilla2019moore}
	{\sc R.~Quintanilla}, {\em Moore--gibson--thompson thermoelasticity},
	Mathematics and Mechanics of Solids, 24 (2019), pp.~4020--4031.
	
	\bibitem{Rudin1987real}
	{\sc W.~Rudin}, {\em Real and Complex Analysis}, McGraw-Hill Book Company.
	
	\bibitem{salsa2016partial}
	{\sc S.~Salsa}, {\em Partial Differential Equations in Action: from Modelling
		to Theory}, vol.~99, Springer, 2016.
	
	\bibitem{staffans1976inequality}
	{\sc O.~J. Staffans}, {\em An inequality for positive definite volterra
		kernels}, Proceedings of the American Mathematical Society, 58 (1976),
	pp.~205--210.
	
	\bibitem{thompson}
	{\sc P.~Thompson}, {\em Compressible Fluid Dynamics}, McGraw-Hill, New York,
	NY, 1972.
	
	\bibitem{van2021existence}
	{\sc K.~Van~Bockstal}, {\em Existence of a unique weak solution to a
		non-autonomous time-fractional diffusion equation with space-dependent
		variable order}, Advances in Difference Equations, 2021 (2021), pp.~1--43.
	
	\bibitem{vergara2008lyapunov}
	{\sc V.~Vergara and R.~Zacher}, {\em Lyapunov functions and convergence to
		steady state for differential equations of fractional order}, Mathematische
	Zeitschrift, 259 (2008), pp.~287--309.
	
	\bibitem{vergara2015optimal}
	\leavevmode\vrule height 2pt depth -1.6pt width 23pt, {\em Optimal decay
		estimates for time-fractional and other nonlocal subdiffusion equations via
		energy methods}, SIAM Journal on Mathematical Analysis, 47 (2015),
	pp.~210--239.
	
	\bibitem{wismer1995explicit}
	{\sc M.~G. Wismer and R.~Ludwig}, {\em An explicit numerical time domain
		formulation to simulate pulsed pressure waves in viscous fluids exhibiting
		arbitrary frequency power law attenuation}, IEEE transactions on ultrasonics,
	ferroelectrics, and frequency control, 42 (1995), pp.~1040--1049.
	
	\bibitem{zacher2010giorgi}
	{\sc R.~Zacher}, {\em De {G}iorgi--{N}ash--{M}oser estimates for evolutionary
		partial integro-differential equations},  (2010).
	
\end{thebibliography}
\end{document}